\documentclass[12pt]{article}
\usepackage[top=1in,bottom=1in,left=1 in,right=.8in,footskip=0.4in]{geometry}
\usepackage[mathscr]{eucal}
\usepackage{mathtools}
\usepackage{xfrac}
\usepackage{xcolor}
\usepackage{color,soul}
\usepackage{hyperref}
\usepackage{authblk}
\usepackage{orcidlink}
\usepackage[T1]{fontenc}
\usepackage{tikz}
\usepackage{tikz-cd}
\usetikzlibrary{decorations.markings} % for arrows in the middle
\usepackage{enumerate}
\usetikzlibrary{graphs}
\usetikzlibrary{graphs.standard}
\usetikzlibrary{arrows.meta}
\usetikzlibrary{decorations.pathmorphing}
\usepackage{caption}
\usepackage{multicol}
\usepackage{calc}
\usepackage{float}
\usetikzlibrary{calc}
\tikzstyle{vertex}=[circle, draw, fill=black, inner sep=0pt, minimum size=3pt]
\tikzstyle{grvert}=[circle, draw, gray, fill=gray,  inner sep=0pt, minimum
size=6pt]
\usepackage{layout}
\usepackage{graphicx}
\usepackage{amsfonts,amsmath,amssymb,amsthm}
\hypersetup{colorlinks=true,linkcolor=red,citecolor=blue}
\newtheorem{theorem}{Theorem}[section]
\newtheorem{lemma}[theorem]{Lemma}
\newtheorem{corollary}[theorem]{Corollary}
\newcommand{\W}{\mathcal{W}}
\newcommand{\inn}{\mathbf{i}}
\newcommand{\trr}{\mathbf{t}}
\newcommand{\T}{\mathcal{T}}

\newcommand{\V}{\mathcal{V}}
\newcommand{\K}{\mathcal{K}}
\newcommand{\s}{\mathcal{S}}
\newcommand{\D}{\mathcal{D}}
\newcommand{\C}{\mathcal{C}}

\newcommand{\cl}{\operatorname{cl}}
\newcommand{\Bd}{\operatorname{Bd}}
\newcommand{\crit}{\operatorname{Crit}}
\newcommand{\Id}{\operatorname{Id}}
\newcommand{\Sd}{\operatorname{Sd}}
\theoremstyle{definition}
\newtheorem{definition}[theorem]{Definition}
\newtheorem{observation}[theorem]{Observation}
\newtheorem{remark}[theorem]{Remark}
\theoremstyle{plain}
\newtheorem{proposition}[theorem]{Proposition}

\title{Combinatorial degree version of a generalized $\mathbb{Z}_p$-Tucker's lemma with a combinatorial proof}
	
\author[a]{Sajal Mukherjee~\orcidlink{0009-0004-6959-4334}\thanks{\texttt{sajal.mukherjee@tcgcrest.org}}}

\author[b]{Pritam Chandra Pramanik~\orcidlink{0009-0000-8332-3347}\thanks{\texttt{pritam.pramanik.80@tcgcrest.org}}}

\affil[a,b]{\small Institute for Advancing Intelligence (IAI), TCG CREST (affiliated with the Academy of Scientific \& Innovative Research (AcSIR), India), Kolkata -- 700091, West Bengal, India}
\affil[b]{\small Department of Mathematics, National Institute of Technology (NIT) Durgapur, Durgapur -- 713209, West Bengal, India}

\date{}
	
\begin{document}
	\maketitle
	
	\begin{abstract}
		Combinatorial analogues of classical Borsuk--Ulam-type theorems (e.g., Tucker's lemma, $\mathbb{Z}_p$-Tucker's lemma, etc.) have numerous important applications in combinatorics. In this paper, we formulate a combinatorial degree version of a generalized  $\mathbb{Z}_p$-Tucker's lemma. Our proof is purely combinatorial in the sense that it does not involve homology, cohomology or any other notions from continuous topology. In order to prove the aforementioned degree theorem, as a main technical tool, we prove a Hopf trace-type formula, which is also purely combinatorial and involves no homology. This combinatorial Hopf trace formula is of independent interest.
		
		\textbf{Keywords:} simplicial complex, discrete Morse theory, Tucker's lemma, Hopf trace formula, mapping degree.
		
		\textbf{MSC 2020:} 57Q70 (primary), 05E45, 55U15.
	\end{abstract}

	\section{Introduction}
    Combinatorial formulations of theorems in classical algebraic topology, especially the Borsuk--Ulam, or in general Dold-type results (e.g., Todd and Freund's combinatorial treatment of Tucker's lemma \cite{Freund}, Ziegler's $\mathbb{Z}_p$-Tucker's lemma \cite{Ziegler} etc.) have been fruitful in obtaining clean combinatorial proofs of some of the famous results in combinatorics and graph theory, such as Matousek's proof of Kneser's conjecture~\cite{Matousek}, and some of its far reaching generalizations due to Ziegler~\cite{Ziegler}, to name a few. See also \cite{hanke,meunier, simplotopal} for  combinatorial treatments of several Borsuk--Ulam-type theorems and their applications in combinatorics.
    
    In this paper, we formulate a combinatorial degree version of Ziegler's $\mathbb{Z}_p$-Tucker's lemma~\cite{Ziegler} in a more generalized form and provide a purely combinatorial proof. Here we use the term `combinatorial' in the sense of Ziegler~\cite{Ziegler}, i.e., our formulation and proof do not involve homology or any other notions from continuous topology. In order to state the result, we need to borrow some terminologies from discrete Morse theory (see \cite{forman1998, forman2002}). However, it is worth mentioning that throughout the paper we use only some terminologies and not any theorems from discrete Morse theory. 
    
    	Let $\K$ be a $d$-dimensional simplicial complex. A \emph{discrete vector field} $\V$ on  $\K$ is a collection of ordered pairs of simplices of the form $(\alpha,\beta)$, such that (i) $\alpha$ is a facet of $\beta$, and (ii) each simplex of $\K$ is in \emph{at most} one pair of $\V$.
    
    \begin{definition}[$\V$-trajectory]
        Let $\K$ be a $d$-dimensional simplicial complex and $\V$ be a discrete vector field on it. Let $\beta_0, \beta_r \in \K$ such that,  $\dim(\beta_0)=\dim(\beta_r)=q$, for some $q\in\{0, \ldots, d\}$. A \emph{$\V$-trajectory} (or simply, a \emph{trajectory}, when the discrete vector field is clear from context)
    	from $\beta_0$ to $\beta_r$, is an alternating sequence of $q$ and $(q-1)$ dimensional simplices,
    	\[P:\beta_0^{(q)}, \alpha_1^{(q-1)}, \beta_1^{(q)}, \ldots, \alpha_r^{(q-1)},
    	\beta_r^{(q)},\]
    	such that, for each $i \in \{1,\ldots,r\}$, $(\alpha_i,\beta_i) \in
    	\V$, and $\beta_{i-1} (\neq \beta_i) \supsetneq \alpha_i$.
    \end{definition}	
    The simplices $\beta_0$ and $\beta_r$ are called the \emph{initial simplex} (denoted by $\inn(P)$) and \emph{terminal simplex} (denoted by $\trr(P)$) of the trajectory $P$, respectively.

    The \emph{weight} $P$ is defined by,
    \[w(P):=\prod_{i=1}^{r}(-[ \beta_{i-1},\alpha_i] [ \beta_i, \alpha_i ]),\]
    
    where, for any $\sigma, \tau \in \K$, $[ \sigma,\tau]$ is the incidence number between between $\sigma$ and $\tau$. When $r=0$, $P$ is said to be \emph{trivial}, and its weight is defined to be $1$. The trajectory $P$ is said to be \emph{non-trivial closed} if
    $r > 0$ and $\beta_{r} = \beta_0$.
    \begin{definition} [co-$\V$-trajectory]
    	Let $\K$ be a $d$-dimensional simplicial complex and $\V$ be a discrete vector field on it. Let $\beta_0, \beta_r \in \K$ such that,  $\dim(\beta_0)=\dim(\beta_r)=q$, for some $q\in\{0, \ldots, d\}$.
    	A \emph{co-$\V$-trajectory} (or simply, a \emph{co-trajectory}) from $\beta_0$ to $\beta_r$ is an alternating sequence of $q$ and $(q+1)$ dimensional simplices,
    	\[Q:\beta_0^{(q)}, \tau_1^{(q+1)}, \beta_1^{(q)}, \ldots, \tau_r^{(q+1)},
    	\beta_{r}^{(q)},\]
    	such that, for each $i \in \{1,\ldots,r\}$, $(\beta_i,\tau_i) \in \V$
    	and $\beta_{i-1} (\neq \beta_i)\subsetneq \tau_{i}$.
    \end{definition}

    We now introduce terminology for co-trajectories analogous to that used for trajectories. The simplices $\beta_0$ and $\beta_r$ are defined to be the \emph{initial simplex} (denoted by $\inn(Q)$) and \emph{terminal simplex} (denoted by $\trr(Q)$) of $Q$, respectively. 
    
    The \emph{weight} of $Q$ is defined by,
    \[w(Q):=\prod_{i=1}^{r}(-[ \tau_i, \beta_{i-1}] [ \tau_i
    ,\beta_i ]).\]
    
    When $r=0$, $Q$ is said to be \emph{trivial}, and its weight is defined to be $1$. The co-trajectory $Q$ is said to be \emph{non-trivial closed} if
    $r > 0$ and $\beta_{r} = \beta_0$.
    \begin{definition}[Gradient vector field]
    	A \emph{gradient vector field} $\V$ on a simplicial complex $\K$ is a discrete vector field on $\K$ which does not admit any non-trivial closed $\V$-trajectories.
    \end{definition}
        We remark that a gradient vector field $\V$ is also referred to as an \emph{acyclic
    	matching} (or \emph{pairing}) in literature (see~\cite{chari}), and the property that $\V$ does not admit a non-trivial closed $\V$-trajectory is called \emph{acyclicity} of $\V$.
    \begin{definition}[Critical simplex]
    	Let $\K$ be a simplicial complex and $\V$ be a gradient vector field on it. A non empty simplex $\sigma$ is said to be a \emph{$\V$-critical simplex} (or simply \emph{critical simplex}, when the gradient vector field is clear from context) if one of the following holds:
    	\begin{enumerate}[(i)]
    		\item $\sigma$ does not appear in any pair of $\V$, or
    		\item $\sigma$ is a $0$-simplex and $(\emptyset,\sigma) \in \V$.
    	\end{enumerate}
    \end{definition}
    For $0 \leq q \leq d$, the set of all $q$-dimensional $\V$-critical simplices in a simplicial complex $\K$ is denoted by $\crit_q^{(\V)}(\K)$.

    \begin{definition}[Critical chain and co-critical chain] \label{D1.6}
    	Let $\K$ be a $d$-dimensional simplicial complex and $\V$ be a gradient vector field on it. Let $\sigma \in \crit_q^{(\V)}(\K)$, for some $q \in \{0,\ldots,d\}$. Then the \emph{critical chain} ($\overrightarrow{\sigma}^{(\V)}$) and the \emph{co-critical chain} ($\overleftarrow{\sigma}^{(\V)}$), corresponding to $\sigma$, are defined by
    	\[\overrightarrow{\sigma}^{(\V)}:= \sum_{\substack{P: P \text{ is a}\\ \V\text{-trajectory},\\ \inn(P)=\sigma}} w(P)\cdot \trr(P),~\text{and  }~ \overleftarrow{\sigma}^{(\V)}:= \sum_{\substack{Q: Q \text{ is a}\\ \text{co-}\V\text{-trajectory},\\ \inn(Q)=\sigma}} w(Q)\cdot \trr(Q),\]
    	respectively.	
     \end{definition} 
    	A \emph{$d$-dimensional pseudomanifold} is a $d$-dimensional pure simplicial complex (a simplicial complex in which all the maximal simplices are of the same dimension), such that, (i) any $(d-1)$-simplex is contained in exactly two maximal simplices, and
    	(ii) for any two maximal simplices, say $\alpha$ and $\beta$, there exists a sequence of  maximal simplices,
    	$ (\alpha=)~\alpha_0, \alpha_1, \ldots, \alpha_n~(=\beta)$, 
    	such that, for $i\in \{0,\ldots, n-1\}$, $\alpha_i\cap \alpha_{i+1}$ is a $(d-1)$-simplex. 
   
    \begin{definition}[Combinatorial sphere]
    	A $d$-dimensional pseudomanifold $\s$ is said to be a \emph{combinatorial $d$-sphere} (or simply, a \emph{combinatorial sphere}), if there exists a gradient vector field $\V$ on $\s$ having exactly two critical simplices, one of dimension $0$ and one of dimension $d$.
    \end{definition}
    
    Now, we are in a position to state the main results of this paper. Let $p$ be a prime, and  $\s$ be a $\mathbb{Z}_p$-combinatorial sphere, that is, $\s$ is a combinatorial sphere and the group $\mathbb{Z}_p$ acts on $\s$ freely and simplicially. Then we have the following.
    \begin{theorem} \label{t1.7}
    	Let $\s$ be a $\mathbb{Z}_p$-combinatorial sphere, and $\Bd^k(\s)$ be the $k$-th barycentric subdivision of $\s$. Then, for any  $\mathbb{Z}_p$-equivariant simplicial map, $f: \operatorname{Bd}^{k}(\s) \rightarrow \s $, 
    	\[\deg(f)\equiv  1 \pmod{p}.\] 
    \end{theorem}
    In the above theorem, $\deg(f)$ is a direct combinatorial analogue (see Definition~\ref{d4.6}) of mapping degree in topology.
    
    Now let us state Ziegler's $\mathbb{Z}_p$-Tucker's lemma (see \cite{Ziegler}) in a slightly generalized form. Let $m\in \mathbb{N}$, and $\s_1, \s_2, \ldots, \s_m$ be $\mathbb{Z}_p$-combinatorial spheres (not necessarily of the same dimension). Let $x*\Bd^k(\s_1*\s_2*\ldots\s_m)$ be the cone over $\Bd^k(\s_1*\s_2*\ldots\s_m)$ with the apex $x$. Here `$*$' denotes the simplicial join (see Definition~\ref{join}). Then the $\mathbb{Z}_p$-Tucker's lemma essentially goes as follows.
     \begin{theorem}[Generalized $\mathbb{Z}_p$-Tucker's lemma] \label{t1.8}
    	Let $m\in \mathbb{N}$, and $\s_1, \s_2, \ldots, \s_m$ be $\mathbb{Z}_p$-combinatorial spheres. Then, there does not exists any simplicial map $f: x* \Bd^k(S_1* \cdots *S_m) \rightarrow S_1 * \cdots *S_m$, such that,
    	\[f\big|_{\Bd^k(S_1* \cdots *S_m)}:\Bd^k(S_1* \cdots S_m) \rightarrow S_1 * \cdots *S_m\]
    	is $\mathbb{Z}_p$-equivariant.
    \end{theorem}
    
    In Section~\ref{s4} (Proposition~\ref{p4.12}) we prove that join of two combinatorial spheres is also a combinatorial sphere. So, Proposition~\ref{p4.12} and Theorem~\ref{t1.7} together lead to the following combinatorial degree version of the generalized $\mathbb{Z}_p$-Tucker's lemma.
    
     \begin{theorem}[Combinatorial degree version of generalized $\mathbb{Z}_p$-Tucker's lemma] \label{t1.9} Let $m \in \mathbb{N}$,
    	and $\s_1, \s_2, \ldots, \s_m$ be $\mathbb{Z}_p$-combinatorial spheres. Then, for any $\mathbb{Z}_p$-equivariant simplicial map 
    	$f: \Bd^{k}(\s_1 * \s_2 * \cdots *\s_m) \rightarrow (\s_1 * \s_2 * \cdots *\s_m),$
    	\[\deg(f)\equiv  1 \pmod{p}.\] 
    \end{theorem}    
    In Appendix~\ref{Appendix}, we show that Theorem~\ref{t1.8} follows form Theorem~\ref{t1.9}.
    
    As the main technical tool to prove our main result (Theorem~\ref{t1.7}), we develop the following combinatorial version of Hopf trace formula, which is of independent interest.
    
     \begin{theorem}[Combinatorial Hopf trace formula]\label{t1.10}
    	Let $\K$ be a $d$-dimensional simplicial complex with a gradient vector field $\V$  on it. Let, for any $q$, with $0\leq q\leq d$, and any two $q$-chains $\xi_1, \xi_2\in C_q(\K, \mathbb{Z})$, such that $\xi_1= \sum_{i=1}^n s_i\sigma_i$ and $\xi_2= \sum_{i=1}^n t_i\sigma_i$, where $\{\sigma_1,\sigma_2, \ldots, \sigma_n\}$ is the set of all $q$-dimensional simplices of $\K$, the inner product $\langle \xi_1,\xi_2\rangle$ be $\sum_{i=1}^n s_it_i$. Let $\phi_\#:C_\#(\K,\mathbb{Z}) \rightarrow C_\#(\K,\mathbb{Z})$ be a chain map. Then
    	\begin{equation} \label{eq-main}
    		\sum_{q=0}^d (-1)^q \operatorname{tr}(\phi_q)=\sum_{q=0}^d (-1)^q \sum_{\sigma\in \crit_q^{(\V)}(\K)} \langle \overleftarrow{\sigma}^{(\V)},\phi_q(\overrightarrow{\sigma}^{(\V)})
    		\rangle.
    	\end{equation}
    \end{theorem}
    
    It is worth mentioning that, this is a purely combinatorial identity and unlike the Hopf trace formula (see \cite{bredon, munkres}), Theorem~\ref{t1.10} does not involve any homology and holds for any coefficient ring $R$, not just fields (although, throughout this paper we use $\mathbb{Z}$ as the coefficient ring, our proof of Theorem~\ref{t1.10} works for any ring as well). 
    
	\section{Preliminaries}
	A \textit{directed graph} (or \textit{digraph}) is an orderd pair $G=(V,E)$, where $V$ is a finite nonempty set, and $E\subseteq\{(a,b): a,b\in V \text{ and }a\neq b\}$. The set $V$ is called the \textit{vertex set} of the digraph $G$, and the elements of the set $E$ are called \emph{directed edges} (or simply, \emph{edges}). A \textit{directed cycle} in $G$ is a sequence of vertices, $a_1,a_2,\ldots, a_n$, where each $a_i$ is distinct for $i\in \{1,\ldots,n-1\}$, and $a_n=a_1$, and for each $i\in \{1,\ldots,n-1\}$, $(a_i,a_{i+1})\in E$. A digraph is said to be a \text{directed acyclic graph} if it does not contain any directed cycle. A \textit{topological ordering} of a digraph is a linear ordering of its vertices such that for each directed edge $(a,b)\in E$, $a$ comes before $b$ in the ordering. We have the following well-known theorem.
	\begin{theorem}[Topological sorting \cite{cormen}] \label{t-tp}
		Every directed acyclic graph has a topological ordering, and vice-versa. 
	\end{theorem}
	 
	An \emph{abstract simplicial complex} $\K$ (or in short, a \emph{simplicial complex}) is a finite non-empty collection of finite sets with the property that if $\sigma \in \K$ and $\tau \subseteq \sigma$, then $\tau \in \K$. Note that, the empty set $\emptyset$ always belongs to $\K$. 
	
	 An element in $\K$ is called a \emph{simplex}. The \text{dimension} of a simplex $\sigma$ is defined as, $\dim(\sigma):=|\sigma|-1$. If $\dim(\sigma)=q$, we call $\sigma$ a $q$-\emph{simplex}. We often denote a $q$-simplex $\sigma$ by $\sigma^{(q)}$. The dimension of a simplicial complex $\K$ is defined as, $\dim(\K):=\max\{\dim(\sigma):\sigma\in\K\}$. For $\tau,\sigma \in \K$, if $\tau \subseteq \sigma$, then $\tau$ is called a \emph{face} of $\sigma$. If $\tau$ is a face of $\sigma$ with $\dim({\tau})=\dim({\sigma})-1$, then $\tau$ is said to be a \emph{facet} of $\sigma$. The \emph{vertex set} of $\K$, denoted by $V(\K)$, is defined as, $V(\K)= \cup_{\sigma \in \K} ~\sigma$. For simplicity, any $v\in V(\K)$, we represent the simplex $\{v\}$ as $v$. A simplex in a simplicial complex is said to be \emph{maximal}, if it is not contained in any other simplex. Note that, for a maximal simplex $\sigma \in \K$, $\K \setminus \{\sigma \}$ is also a simplicial complex. A simplicial complex is called \emph{pure}, if all of its maximal simplices are of the same dimension. A subset $\K^\prime$ of $\K$ is called a \emph{subcomplex} of $\K$ if $\K^\prime$ is also a simplicial complex. For a simplex $\sigma \in \K$, the \emph{closure of $\sigma$}, denoted by $\cl(\sigma)$, is defined as $\cl(\sigma):= \{\tau \in \K: \tau \subseteq \sigma\}$. Note that, $\cl(\sigma)$ is a simplicial complex itself.
	 
	  For, $n\in \mathbb{N}$, $\Delta^n$ is the simplicial complex which contains all the subsets of a set of cardinality $n$. For some non-negative integer $q$ with $0\leq q \leq n$, the $q$-skeleton of $\Delta^n$ is defined as $\Delta^n_q:= \{\sigma \in \Delta^n : \dim(\sigma) \leq q\}$.
	 Note that, $\Delta^n_q$ is a subcomplex of $\Delta^n$.

	 \begin{definition}[Join of two simplicial complexes] \label{join}
	 	Let $\K_1$ and $\K_2$ be two simplicial complexes, such that $V(\K_1)\cap V(\K_2)=\emptyset$. Then, the \emph{join} of $\K_1$ and $\K_2$, denoted by $\K_1*\K_2$, is the simplicial complex defined as,
	 	\[\K_1*\K_2:=\{\sigma \sqcup \tau : \sigma \in \K_1 \text{ and } \tau \in \K_2\}.\]
	 \end{definition}
        Note that, if $\dim(\K_1)=m$ and $\dim(\K_2)=n$, then $\dim(\K_1 * \K_2)= m+n+1$. 
    \begin{definition}[Cone over a simplicial complex]
    	Let $\K$ be a simplicial complex and $x \notin V(\K)$. Then,  the \emph{cone} over $\K$ (with the apex $x$), is defined to be the simplicial complex $\{\emptyset, \{x\}\}*\K$, and denoted by $x*\K$. 
    \end{definition}
        
     \begin{definition}[Barycentric subdivision of a simplicial complex] The
    	barycentric subdivision of a simplicial complex $\K$, denoted by $\Bd(\K)$, is the simplicial complex whose vertices correspond to the non-empty simplices of $\K$, and simplices of $\Bd{(\K)}$ correspond to the chains (under set inclusion) of non-empty simplices of $\K$. In other words, for each $q\geq 0$, the set of all $q$-simplices of  $\Bd(\K)$ are as follows.
    	\[S_q(\Bd(\K)):=\{\{v_{\sigma_0},\ldots,v_{\sigma_q}\}: \sigma_i \in \K \setminus \{\emptyset\} \text{ for all } i\in \{1,\ldots, q\}, \text{ and } \sigma_0 \subsetneq \sigma_1 \subsetneq \cdots \subsetneq \sigma_q\}.\] 
    \end{definition}
    The $k$-th barycentric subdivision of $\K$, denoted by $\Bd^k(\K)$, is defined recursively as  $\Bd^k(\K):=\Bd(\Bd^{k-1}(\K))$, with $\Bd^0{(\K)}=\K$.
	\begin{definition}[Simplicial map]
		Let $\K_1$ and $\K_2$ be two simplicial complexes. A \emph{simplicial map} $f$ from $\K_1$ to $\K_2$ is a function $f: V(\K_1) \rightarrow V(\K_2)$, such that, for any $\sigma \in \K_1$, $f(\sigma) \in \K_2$, where, for any simplex $\sigma=\{v_1,\ldots,v_n\}$ in $\K_1$, $f(\sigma)$= $\{f(v_1),\ldots,f(v_n)\}$.
	\end{definition}

	An \emph{orientation} of a simplex is an ordering of its vertices, with two orderings defining the same orientation if and only if they differ by an even permutation. A simplex with an orientation is said to be a \emph{oriented simplex}. We denote an oriented $q$-simplex $\sigma$ consisting of the vertices $x_{i_0}, x_{i_1}, \ldots , x_{i_q}$, with $i_0 < i_1 < \ldots < i_q$, by $(x_{i_0}~x_{i_1} \ldots x_{i_k})$.

	\begin{definition}[Incidence number between simplices]
		Let $\sigma ^{(q)}=(x_0~x_1\ldots x_q)$ and $\tau^{(q-1)}$ are oriented simplices of a simplicial complex $\K$. Then the \emph{incidence number} between $\sigma$ and $\tau$ is defined by,
		\[[\sigma^{(q)},\tau^{(q-1)}]: =
		\begin{cases}
			(-1)^i, & \text{ if } \tau\subsetneq \sigma \text{ and } \tau=(x_0\dots\widehat{x}_i\ldots x_{q}), \\ 
			0, & \text{ otherwise },	
		\end{cases}\]
		where
		$(x_0\ldots\widehat{x}_i\ldots x_q)$ is the oriented $(k-1)$-simplex
		obtained from $\sigma$ after deleting $x_i$.
	\end{definition}
   
	For some $q\in \mathbb{N}$, let $S_q(\K)$ (or simply, $S_q$, when the simplicial complex is clear from context) be the set of all $q$-simplices of $\K$ and $C_q(\K,\mathbb{Z})$ be the free $\mathbb{Z}$-module generated by $S_q(\K)$. An element of $C_q(\K,\mathbb{Z})$ is called \emph{$q$-chain} (or simply, a \emph{chain}). If $\K$ is a $d$-dimensional simplicial complex, then $C_q(\K,\mathbb{Z})=\{0\}$, whenever $q<0$ or $q>d$. We define the \emph{$q$-th boundary map}  $\partial_q:C_{q}(\K,\mathbb{Z})\rightarrow C_{q-1}(\K,\mathbb{Z})$ (upon linearly extending the domain of the following map from $S_q(\K)$ to $C_{q}(\K, \mathbb{Z})$) as follows.
	\[\partial_q(\sigma)=\sum_{\tau\in S_{q-1}(\K)}[\sigma,\tau ] \cdot \tau.\]
	A $q$-chain $\sigma$ is said to be a \emph{$q$-cycle} (or simply, a \emph{cycle}) if $\partial_q(\sigma)=0$. The set of all $q$-cycles is denoted by $Z_q(\K, \mathbb{Z})$.
    The sequence,
    \[\cdots \xrightarrow{\partial_{q+2}} C_{q+1}(\K,\mathbb{Z})\xrightarrow{\partial_{q+1}} C_{q}(\K,\mathbb{Z})\xrightarrow{\partial_q}C_{q-1}(\K,\mathbb{Z}) \xrightarrow{\partial_{q-1}} \cdots\]
	is called \emph{simplicial chain complex} of $\K$, and denoted by $(C_\#{(\K,\mathbb{Z})},\partial_\#)$.

	We recall that, a  chain map $\phi_\#$ between two chain complexes $(M_\#,\partial^M_{\#})$ and
	 $(N_\#,\partial^N_{\#})$ is a sequence of homomorphisms $\phi_q:M_q\rightarrow N_q$,   such that, $\partial^N_{q} \circ  \phi_q=\phi_{q-1} \circ \partial^M_{q}$, for
	 all $q\in \mathbb{N}$. That means, the following diagram commutes.
	 \[\begin{tikzcd}
	 		\cdots && {M_{q+1}} && {M_q} && {M_{q-1}} && \cdots \\
	 		\\
	 		\cdots && {N_{q+1}} && {N_{q}} && {N_{q-1}} && \cdots
	 		\arrow["{\partial^M_{q+2}}", from=1-1, to=1-3]
	 		\arrow["{\partial^M_{q+1}}", from=1-3, to=1-5]
	 		\arrow["{\phi_{q+1}}", from=1-3, to=3-3]
	 		\arrow["{\partial^M_{q}}", from=1-5, to=1-7]
	 		\arrow["{\phi_q}", from=1-5, to=3-5]
	 		\arrow["{\partial^M_{q-1}}", from=1-7, to=1-9]
	 		\arrow["{\phi_{q-1}}", from=1-7, to=3-7]
	 		\arrow["{\partial^N_{q+2}}", from=3-1, to=3-3]
	 		\arrow["{\partial^N_{q+1}}", from=3-3, to=3-5]
	 		\arrow["{\partial^N_{q}}", from=3-5, to=3-7]
	 		\arrow["{\partial^N_{q-1}}", from=3-7, to=3-9]
	 	\end{tikzcd}\]
	 Let $f:\K_1 \rightarrow \K_2$ be a simplicial map. For each $q\in \mathbb{N}$, we can define homomorphisms $f_q:C_q(\K_1, \mathbb{Z})\rightarrow C_q(\K_2,\mathbb{Z})$ as follows. 
	 \[f_q(\sigma)=
	 \begin{cases}
	 	f(\sigma), & \text{ if }\dim(f(\sigma))=\dim(\sigma),\\
	 	0, &  \text{ if }\dim(f(\sigma))<\dim(\sigma),
	 \end{cases}\]
 where $\sigma \in S_q(\K_1)$. Moreover, the following diagram is commutative.
 	\[\begin{tikzcd}
 	\cdots && {C_{q+1}(\K_1,\mathbb{Z})} && {C_q(\K_1,\mathbb{Z})} && {C_{q-1}(\K_1,\mathbb{Z})} && \cdots \\
 	\\
 	\cdots && {C_{q+1}(\K_2,\mathbb{Z})} && {C_{q}(\K_2,\mathbb{Z})} && {C_{q-1}(\K_2,\mathbb{Z})} && \cdots
 	\arrow["{\partial_{q+2}}", from=1-1, to=1-3]
 	\arrow["{\partial_{q+1}}", from=1-3, to=1-5]
 	\arrow["{f_{q+1}}", from=1-3, to=3-3]
 	\arrow["{\partial_q}", from=1-5, to=1-7]
 	\arrow["{f_q}", from=1-5, to=3-5]
 	\arrow["{\partial_{q-1}}", from=1-7, to=1-9]
 	\arrow["{f_{q-1}}", from=1-7, to=3-7]
 	\arrow["{\partial_{q+2}}", from=3-1, to=3-3]
 	\arrow["{\partial_{q+1}}", from=3-3, to=3-5]
 	\arrow["{\partial_{q}}", from=3-5, to=3-7]
 	\arrow["{\partial_{q-1}}", from=3-7, to=3-9]
 \end{tikzcd}\]

  Hence, any simplicial map $f:\K_1 \rightarrow \K_2$ induces a chain map $f_\#:C_\#(\K_1,\mathbb{Z})\rightarrow C_\#(\K_2,\mathbb{Z})$.
  
   For $k \in \mathbb{N}$, there is a well-known chain map $\Sd^k_\#: C_\#(\s,\mathbb{Z}) \rightarrow C_\#(\Bd^k(\s),\mathbb{Z})$, called the \emph{$k$-th subdivision map}. We refer to \cite{bredon} and \cite{munkres} for details. 

  We now define \emph{inner product} of chains. Let $\K$ be a $d$-dimensional simplicial complex. For some $0\leq q \leq d$, let $B= \{b_1,b_2,\ldots b_m\}$ be a $\mathbb{Z}$-basis of $C_q(\K,\mathbb{Z})$. Let $\xi_1,\xi_2 \in C_q(\K, \mathbb{Z})$ be two chains, such that $\xi_1= \sum_{i=1}^m s_i b_i$ and  $\xi_2= \sum_{i=1}^m s_i^\prime b_i$  (here $s_i, s_i^\prime \in \mathbb{Z}$, for $i\in \{1,2, \ldots,m\}$). We define the \emph{inner product} of $\xi_1$ and $\xi_2$, with respect to the basis $B$, as
  $\langle \xi_1,\xi_2 \rangle_{B}:= \sum_{i=1}^m s_is^\prime_i$.
  When $B=S_q$ (the standard $\mathbb{Z}$-basis of $C_q{(\K, \mathbb{Z})}$), we often omit the subscript $B$ from  $\langle \xi_1,\xi_2 \rangle_{B}$, and simply write it as $\langle \xi_1,\xi_2 \rangle$. Note that, for any $\xi \in C_q(\K,\mathbb{Z})$, $\langle \xi,b_i \rangle_B$ is the coefficient of $b_i$ in $\xi$ with respect to the basis $B$. For any chain $\xi \in C_q(\K,\mathbb{Z})$, we define $\operatorname{support}(\xi):= \{\sigma \in S_q(\K): \langle \xi, \sigma \rangle \neq 0 \}$.
  
  We recall the following criteria of replacing an element of a $\mathbb{Z}$-basis to obtain a new $\mathbb{Z}$-basis. 
   \begin{theorem}[Replacement theorem for free $\mathbb{Z}$-modules] \label{treplacement} Let $C$ be a free $\mathbb{Z}$-module and $B=\{b_1, \ldots, b_m\}$ be a $\mathbb{Z}$-basis of $C$. Let $\xi\in C$, such that, $\xi= \sum_{i=1}^{m} s_i b_i$, where $s_i \in \mathbb{Z}$ for all $i \in\{1,\ldots,m\}$. If for some $i \in \{1,\ldots,m\}$, $s_i=\pm 1$ then $B'=\{b_1,\ldots, b_{i-1},\xi, b_{i+1}, \ldots, b_m\}$ is also a $\mathbb{Z}$-basis of $C$.
  \end{theorem}
 
Let $\K$ be a $d$-dimensional simplicial complex and $\V$ be a gradient vector field on it. If $(\alpha^{(q-1)}, \beta^{(q)}) \in \V$, for some $q \in \{0,\ldots, d-1\}$, then sometimes we pictorially represent it as $\alpha \rightarrowtail \beta$ (or $\beta \leftarrowtail \alpha$). Then a $\V$-trajectory,  \[P:\beta_0^{(q)}, \alpha_1^{(q-1)}, \beta_1^{(q)}, \ldots, \alpha_r^{(q-1)},
	\beta_r^{(q)}\]
	is notationally represented as follows.
	
	\[P:\beta_0^{(q)}\rightarrow \alpha_1^{(q-1)}\rightarrowtail
	\beta_1^{(q)}\rightarrow \cdots\rightarrow \alpha_r^{(q-1)}\rightarrowtail
	\beta_r^{(q)}\]
	(here $``\rightarrow"$ represents \emph{set inclusion}).
	Similarly, a co-$\V$-trajectory,

		\[Q:\beta_0^{(q)}, \tau_1^{(q+1)}, \beta_1^{(q)}, \ldots, \tau_r^{(q+1)},
		\beta_{r}^{(q)},\]
	is represented as,
	\[Q:\beta_0^{(q)} \leftarrow \tau_1^{(q+1)}\leftarrowtail
	\beta_1^{(q)}\leftarrow \cdots\leftarrow \tau_r^{(q+1)}\leftarrowtail
	\beta_{r}^{(q)}.\]
     
	If $(\alpha, \beta) \in \V$, we define $\V(\alpha):= \beta$. We now make the following important observations.

    \begin{observation} \label{Ob2.13} Let $\K$ be a simplicial complex and $\V$ be a gradient vector field on it. Then, we have the following observations.
    	\begin{enumerate}[(i)]
    		\item Any subset $\W$ of $\V$ is also a gradient vector field on $
    		\K$.
    		
    		\item  Let, $(\alpha, \beta) \in \V$. Then, then the trivial trajectory $P:\beta$ is the unique trajectory, such that $\inn(P)=\beta$ and $\alpha \subsetneq \trr(P)$, for otherwise it violates the acyclicity $\V$.
    		
    		\item Let $(\alpha, \beta) \in \V$, and $P$ be any non-trivial $\V$-trajectory starting from $\beta$. Then, $(\alpha, \beta) \notin P$, otherwise it violates the acyclicity $\V$. So, for $\W= \V \setminus \{(\alpha, \beta)\}$, we can say that the $\V$-trajectories and the $\W$-trajectories, which starts from $\beta$, are same.
    		
    		\item  For any critical simplex $\sigma$, the trivial trajectory $P: \sigma$, whose weight is $1$, is the only trajectory from $\sigma$ to $\sigma$. So, the coefficient of $\sigma$ in $\overrightarrow{\sigma}^{(\V)}$ (with respect to the basis $S_q$), $\langle  \overrightarrow{\sigma}^{(\V)}, \sigma \rangle =1$.

    	\end{enumerate}
    
    \end{observation}

	Let  $\K$ be a $d$-dimensional simplicial complex and $\V$ be a gradient vector field on it. Let $\crit_q^{(\V)}(\K)$ (or simply, $\crit_q^{(\V)}$, when the simplicial complex is clear from context) be the collection of all $q$-dimensional critical simplices, $U_q^{(\V)}(\K)$ (or simply, $U_q^{(\V)}$) be the set of all $q$-dimensional simplices, which are paired (in $\V$) with a higher dimensional simplex of $\K$, and $D_q^{(\V)}(\K)$ (or simply, $D_q^{(\V)}$) be the set of all $q$-dimensional simplices, which are paired with a lower dimensional simplex of $\K$. Clearly, $S_q(\K)=\crit_q^{(\V)}\sqcup U_q^{(\V)}\sqcup D^{(\V)}_q$. We also denote the set of all $\V$-critical simplices of $\K$ by $\crit^{(\V)}(\K)$ (or simply, $\crit^{(\V)}$), the set of all simplices, which are paired with a higher dimensional simplex of $\K$ by $U^{(\V)}(\K)$ (or simply, $U^{(\V)}$), and the set of all simplices, which are paired with a lower dimensional simplex of $\K$ by $D^{(\V)}(\K)$ (or simply, $D^{(\V)}$). Note that, $\K= \crit^{(\V)}\sqcup U^{(\V)} \sqcup D^{(\V)}$.
    
    \begin{definition}[Simplicial collapse]
	Let $\K_1$ and $\K_2$ be two simplicial complex. If there exists a gradient vector field $\V$ on $\K_1$, such that $\operatorname{Crit}^{(\V)}(\K_1)= \K_2\setminus \{\emptyset\}$, then $\K_2$ is said to be a \emph{simplicial collapse} of $\K_1$.
    \end{definition}
    If $\K_2$ is a simplicial collapse of $\K_1$, then we say $\K_1$ collapses into $\K_2$.
    \begin{definition}[Collapsible simplicial complex]
	A simplicial complex $\K$ is said to be collapsible if there exists a gradient vector field $\V$ on $\K$, such that the only critical simplex of $\K$ corresponding to $\V$ is a $0$-dimensional simplex.  
	\end{definition}
  
  Note that, collapsibility is a transitive property, that is, if $\K_1$ collapses into $\K_2$ and $\K_2$ collapses into $\K_3$, then $\K_1$ collapses into $\K_3$. 
\begin{theorem} \cite{kozlov} \label{T2.21}For any two simplicial complex $\K_1$ and $\K_2$, If $\K_1$ collapses into $\K_2$ then $\operatorname{Bd(\K_1)}$ collapses into $\operatorname{Bd}(\K_2)$.
\end{theorem}

 We need the following well-known lemma (see \cite{scoville}). We include a proof for the sake of completeness.
  \begin{lemma}\label{th-coll}
 	Let $\K$ be a simplicial complex and $x \notin V(\K)$. Then, $x*\K$ collapses into $\K$	if and only if $\K$ is collapsible. 
 \end{lemma}
 \begin{proof}
 	First, let us assume that $x*\K$ collapses into $\K$. Let $\W$ be a gradient vector field on $x*\K$, such that $\crit^{(\W)}(x*\K)=\K \setminus \{ \emptyset\}$. We have, $x*\K=\K \sqcup \{\sigma \cup \{x\}: \sigma \in \K\} $. So, every pair in $\W$ has the form $(\sigma \cup \{x\}, \tau \cup \{x\})$, where $\sigma, \tau \in \K$. We define a gradient vector field $\V$ on $\K$ as follows.
 	\[\V= \{(\sigma,\tau): (\sigma \cup \{x\}, \tau \cup \{x\})\in \W\}.\]
 	The acyclicity of $\V$ follows from the acyclicity of $\W$. Note that, the only critical simplex in $\V$ is a $0$-simplex, which is paired with $\emptyset$. Hence, $\K$ is collapsible.
 	
 	Conversely, suppose that $\K$ is collapsible. Let $\V^\prime$ be a gradient vector field on $\K$, such that, the only critical simplex with respect to $\V$ is a $0$-simplex, and without loss of generality, we assume that it is paired with $\emptyset$. We define a gradient vector field $\W^\prime$ on $x*\K$ as follows.
 	\[\W^\prime=\{(\sigma \cup \{x\}, \tau \cup \{x\}): (\sigma, \tau) \in \V\}.\]
 	The acyclicity of $\W^\prime$ follows from the acyclicity of $\V^\prime$ and $\crit^{(\W^\prime)}(x*\K)=\K\setminus \{\emptyset\}$. Hence, $x*\K$ collapses into $\K$. 	
 \end{proof}
 
 The following immediately follows from Lemma~\ref{th-coll}. 
 \begin{corollary} \label{c-4.4}
 	Let $\K_1$ be a simplicial complex and $\K_2$ be a subcomplex of $\K_1$. Let $x \notin V(\K_1)$. Then, the simplicial complex $\K_1 \cup (x*\K_2)$ collapses into $\K_1$ if and only if $\K_2$ is collapsible (see Figure~\ref{f2}).
 \end{corollary}
 \begin{figure}[h]
 	\centering
 	\begin{tikzpicture}
 		\draw (0,0) circle (2cm);
 		\draw (.7,.7) circle (.7cm);
 		\node [label=right:$x$] at (2.5,2.5) {\tiny$\bullet$};
 		\draw (.3,1.27) -- (2.5,2.5);
 		\draw (1.27,.3) -- (2.5,2.5);
 		\node at (.7,.7) {$\K_2$};
 		\node (a)  at (3,1) {$x*\K_2$};
 		\node at (0,-1) {$\K_1$};
 		\draw [->] (a)-- (1.5,1.5);
 	\end{tikzpicture}
 	\caption{$\K_1 \cup (x*\K_2)$} \label{f2}
 \end{figure}

\begin{definition}[$G$-simplicial complex]
	Let $G$ be a group. A simplicial complex $\K$ is said to be a $G$-simplicial complex if, for each $g \in G$, there exists a simplicial automorphism $\phi_g$ (i.e., $\phi_g$ is a bijective simplicial map with $\phi_g^{-1}$ is also a simplicial), such that,
	\begin{enumerate}[(i)]
		\item for the identity element, $e\in G$, $\phi_e(v)=v$ for all $v \in V(\K)$,
		\item for any $g_1, g_2 \in G$, $\phi_{g_1g_2}= \phi_{g_1} \circ \phi_{g_2}$.
	\end{enumerate}
\end{definition} 
 Let $\K$  be a $G$-simplicial complex. Then, for any $g\in G$ and $\sigma \in \K$, we write $g\cdot \sigma$ for $\phi_{g}(\sigma)$.
 Let $\K_1$ and $\K_2$ be two $G$-simplicial complexes. Then, $\K_1* \K_2$ is also a $G$-simplicial complex, with the natural $G$-action on $\K * \K$ defined by,
 $g\cdot(\sigma \sqcup \tau)= (g\cdot \sigma) \sqcup (g \cdot\tau)$. For a $G$-simplicial complex $\K$, $\Bd(\K)$ is also a $G$-simplicial complex. For more details see~\cite{matousekbt}. 
%\begin{theorem} 
%	Join of two $G$-simplicial complex is also a $G$-simplicial complex.
%\end{theorem} 
%\begin{proof}
%Let $\K_1$ and $\K_2$ be two  $G$-simplicial complex, and for any $g \in  G$ and $\phi_g: V(\K_1) \rightarrow V(\K_1)$
%and $\psi_g: V(\K_2) \rightarrow V(\K_2)$ be the corresponding automorphisms. Then we define $\Theta _g=  \phi_g * \psi_g$, such that, for any $\sigma \sqcup \tau \in \K_1 * \K_2$,
%\[\Theta_g(\sigma \sqcup \tau)= \phi_g(\sigma) \sqcup \psi_g(\tau). \] 
%Note that, for the identity element $e\in G$,
%\begin{align*}
%	\Theta_e(\sigma \sqcup \tau)&= \phi_e(\sigma) \sqcup \psi_e(\tau)\\
%	& =\sigma \sqcup \tau.\\
%\end{align*}
%And, for any $g_1, g_2 \in G$,
%\begin{align*}
%	\Theta_{g_1}\circ \Theta_{g_2} (\sigma \sqcup \tau)&=\Theta_{g_1}( \Theta_{g_2} (\sigma \sqcup \tau))\\
%	&=\Theta_{g_1}( \phi_{g_2}(\sigma) \sqcup \psi_{g_2}(\tau))\\
%	&= \phi_{g_1}(\phi_{g_2}(\sigma) ) \sqcup \psi_{g_1}(\psi_{g_2}(\tau))\\
%	&=  \phi_{g_1g_2}(\sigma) \sqcup \psi_{g_1g_2}(\tau)\\
%	&=\Theta_{g_1g_2}(\sigma \sqcup \tau)
%\end{align*}
%Hence, $\K_1 * \K_2$ is a $G$-simplicial complex.
%\end{proof}
\begin{definition}[G-equivariant simplicial map]
	Let $\K_1$ and $\K_2$ be two $G$-simplicial complexes. A simplicial map $f:\K_1 \rightarrow \K_2$ is said to be a \emph{$G$-equivariant simplicial map} (or simply, a \emph{$G$-equivariant map}), if for any $g\in G$ and $\sigma \in \K_1$, $ f(g\cdot\sigma)= g\cdot f(\sigma)$.
\end{definition}

 \begin{definition}[Orientation on a combinatorial sphere] Let $\K$ be a $d$-dimensional psudomanifold, and $\{\sigma_1, \ldots, \sigma_n\}$ be the set of all  $d$-simplices of $\K$. Then, an $n$-tuple $r=(r_1, \ldots, r_n)$, such that $r_i \in \{\pm 1\}$, is said to be an \emph{orientation} on $\K$, if 
	\[\partial(\sum_{i=1}^n r_i\sigma_i) =0.\]  
	$\K$ is said be \emph{orientable}, if there exists an orientation on $\K$. Every combinatorial sphere is orientable (see Theorem~\ref{A2} in Appendix~\ref{Appendix}).
	Let $\s$ be a combinatorial $d$-sphere. Let $\{\sigma_1, \ldots, \sigma_n\}$ be the set of all  $d$-simplices of $\K$, and $r=(r_1, \ldots, r_n)$ be an orientation on $\s$. Then,  the cycle $\sum_{i=1}^n r_i\sigma_i$ is said to be the \emph{fundamental cycle} of $\s$ corresponding to the orientation $r$.
\end{definition}

	\section{Proof of the combinatorial Hopf trace formula}

     In this section, we prove some useful lemmas and use them to prove the combinatorial Hopf trace formula (Theorem~\ref{t1.10}). 
     
     Let $\K$ be a $d$-dimensional simplicial complex and $\V$ be a gradient vector field on it. For any $0 \leq q \leq d$, let  $\overrightarrow{\crit}_q^{(\V)}:=\{\overrightarrow
    {\sigma}^{(\V)}: \sigma \in \crit_q^{(\V)}\}$, and $\widetilde{U}_q^{(\V)}:=\{\widetilde{\alpha}:\alpha \in U_q^{(\V)}\}$, where  $\widetilde{\alpha}:=\partial(\V(\alpha))$. 
     Note that, there is a bijection between $U_q^{(\V)}$ and $\widetilde{U}_q^{(\V)}$, where $\alpha$ goes to $\widetilde{\alpha}$. Let $\widetilde{S}_q^{(\V)}:=\overrightarrow{\crit}_q^{(\V)} \cup \widetilde{U}_q^{(\V)} \cup D_q^{(\V)}$.
	
	\begin{lemma} \label{l-basis}
		Let $\K$ be a $d$-dimensional simplicial complex and $\V$ be a gradient vector field on it. Then, for each $0 \leq q \leq d$, $\widetilde{S}_q^{(\V)}$ is a $\mathbb{Z}$-basis of $C_q(\K, \mathbb{Z})$ .
	\end{lemma}
	\begin{proof}
		Let $\crit_q^{(\V)}=\{\sigma_1,\dots, \sigma_p\}$, $U_q^{(\V)}=\{\alpha_1,\dots, \alpha_l\}$, and
		$D_q^{(\V)}=\{\beta_1,\dots, \beta_m\}$.
		
		\textbf{Step 1:} Recall that, for $1\leq i \leq p$, \[\overrightarrow{\sigma_i}^{(\V)}= \sum_{\substack{P: P \text{ is a}\\ \V\text{-trajectory},\\ \inn(P)=\sigma_i}} w(P)\cdot \trr(P).\]
		For any $i \in \{1,\ldots,p\}$, $\langle \overrightarrow{\sigma_i}^{(\V)}, \sigma_i\rangle=1$ (by Observation~\ref{Ob2.13}(iv)). Moreover, for any non trivial $\V$-trajectory $P$ which starts from $\sigma_i$, $\trr(P) \in D_q^{(\V)}$. Hence, by repeated use of Theorem~\ref{treplacement}, we can replace each $\sigma_i$ with $\overrightarrow{\sigma_i}^{(\V)}$ in $S_q=\{\sigma_1,\dots,\sigma_p,\alpha_1,\dots,\alpha_l,\beta_1,\dots,\beta_m\}$ to get a new $\mathbb{Z}$-basis,
		\begin{equation} \label{B1}
			B_1=\{\overrightarrow{\sigma_1}^{(\V)},\dots,\overrightarrow{\sigma_p}^{(\V)},\alpha_1,\dots,\alpha_l,\beta_1,\dots,\beta_m\}.
		\end{equation}
		
		\textbf{Step 2:} Now, let us construct a digraph $G$, with $V(G)=\{\alpha_1,\dots,\alpha_l\}$, and $(\alpha_i,\alpha_j)\in E(G)$ if and only if $\alpha_i\subseteq \V(\alpha_j)$. Note that, $G$ is acyclic, for otherwise, any directed cycle $\alpha_{i_0} , \alpha_{i_1} , \dots , \alpha_{i_k}~(=\alpha_{i_0})$ in $G$  gives rise to the following closed $\V$-trajectory.
		\[\V(\alpha_{i_k})^{(q+1)} \rightarrow \alpha_{i_{k-1}}^{(q)} \rightarrowtail \V(\alpha_{i_{k-1}})^{(q+1)} \rightarrow \dots \rightarrow \alpha_{i_0}^{(q)} \rightarrowtail \V(\alpha_{i_0})^{(q+1)}~(=\V(\alpha_{i_k})^{(q+1)}).\]
	    Hence by  Theorem~\ref{t-tp}, we can linearly order the vertices of $G$, such that for every directed edge $(\alpha_i,\alpha_j)$, the vertex $\alpha_i$ comes before $\alpha_j$ in the ordering. Without loss of generality, let $\alpha_1,\alpha_2,\dots, \alpha_l$ be the corresponding linear ordering.
		
		Since, $\alpha_1 \notin \operatorname{Support}(\overrightarrow{\sigma_i}^{(\V)})$, for $i \in \{1,\ldots,p\}$, the coefficient of $\alpha_1$ in the representation of $\widetilde{\alpha}_1$ with respect to the basis $B_1$ is $\pm 1$. We replace $\alpha_1$ with $\widetilde{\alpha}_1$ in $B_1$ to get a new $\mathbb{Z}$-basis
		\begin{equation} \label{B2}
		B_2=\{\overrightarrow{\sigma_1}^{(\V)},\dots,\overrightarrow{\sigma_p}^{(\V)},\widetilde{\alpha}_1,\alpha_2\dots,\alpha_l,\beta_1,\dots,\beta_m\}.
    	\end{equation}	
		Again, we note that, $\alpha_2 \notin \operatorname{Support}(\overrightarrow{\sigma_i}^{(\V)})$, for any $i \in \{1\ldots, p\}$, and, since, $\alpha_2\nsubseteq \V(\alpha_1)$ (note that, by our choice, $\alpha_2$ comes after $\alpha_1$ in our linear ordering), $\alpha_2 \notin \operatorname{Support}(\widetilde{\alpha}_1)$ as well. Hence, the coefficient of $\alpha_2$ in the representation of $\widetilde{\alpha}_2$ with respect to $B_2$ is $\pm1$. So, we replace $\alpha_2$ with $\widetilde{\alpha}_2$ in $B_2$ to get a new $\mathbb{Z}$-basis
		\begin{equation} \label{B3}
			B_3=\{\overrightarrow{\sigma_1}^{(\V)},\dots,\overrightarrow{\sigma_p}^{(\V)},\widetilde{\alpha}_1,\widetilde{\alpha}_2,\dots,\alpha_l,\beta_1,\dots,\beta_m\}.
		\end{equation}	
	    Similarly, we replace each $\alpha_i$ with $\widetilde{\alpha_i}$, for $2<i\leq l$, to get the desired $\mathbb{Z}$-basis
	    \begin{equation*}
	    	\widetilde{S}_q^{(\V)}=\{\overrightarrow{\sigma_1}^{(\V)},\dots,\overrightarrow{\sigma_p}^{(\V)},\widetilde{\alpha}_1,\widetilde{\alpha}_2,\dots,\widetilde{\alpha}_l,\beta_1,\dots,\beta_m\}.
	    \end{equation*}	
	\end{proof}
	\begin{lemma} \label{l3.2}
		Let $\K$ be a $d$-dimensional simplicial complex with a gradient vector field $\V$ on it. For $q \in \{0,\ldots,d\}$, let $\crit_q^{(\V)}=\{\sigma_1,\dots,\sigma_p\}$ and $D_q^{(\V)}=\{\beta_1,\dots,\beta_m\}$. Let $\xi \in Z_q(\K, \mathbb{Z})$, such that 
		$\xi=\sum_{i=1}^p s_i\sigma_i + \sum_{j=1}^m t_j\beta_j$, where $s_i, t_j \in \mathbb{Z}$, for $1 \leq i \leq p$ and $1\leq j \leq m$. Then,
	
		\[\xi=\sum_{i=1}^p s_i\overrightarrow{\sigma_i}^{(\V)}. \]
			 .
	\end{lemma}
		\begin{proof}
			We prove this by induction on the  cardinality of $\V$.
			If $|\V|=0$, all the simplices are critical. Hence, $S_q=\crit^{(\V)}_q= \overrightarrow{\crit}^{(\V)}_q$, $D^{(\V)}_q=\emptyset$, and for any $\sigma_i\in \crit_q^{(\V)}$, $\sigma_i=\overrightarrow{\sigma_i}^{(\V)}$. So, our hypothesis is true for $|\V|=0$.
			
			Let us assume that our hypothesis is true for all the gradient vector fields with cardinality less than $n$. Let $|\V|=n$. Let $(\alpha^{(k-1)},\beta^{(k)})\in\V$, for some $k\in \{1,\ldots,d\}$. Let us define \[\W= \V \setminus \{(\alpha^{(k-1)},\beta^{(k)})\}.\] Note that, $\W$ is also a gradient vector field (by Observation~\ref{Ob2.13}(i)) on $\K$, with $|\W|=n-1$.
			
			\textbf{Case 1.}
			If $q>k$ or $q<k-1$, $\crit^{(\V)}_q=\crit^{(\W)}_q$, $D^{(\V)}_q=D^{(\W)}_q$ and $\overrightarrow{\sigma_i}^{(\V)}=\overrightarrow{\sigma_i}^{(\W)}$, for all $1 \leq i \leq p$. Since our hypothesis is true for $\W$, $\xi=\sum_{i=1}^p s_i\overrightarrow{\sigma_i}^{(\W)} $, and hence, $\xi=\sum_{i=1}^p s_i\overrightarrow{\sigma_i}^{(\V)} $.
			
			\textbf{Case 2.}
			If $q=k$, then $\beta=\beta_i$, for some $i\in \{1,\ldots,m\}$. Without loss of generality, let $\beta=\beta_1$. Then $\crit^{(\W)}_q=\crit^{(\V)}_q \cup \{\beta_1\}$ and $D^{(\W)}_q=D^{(\V)}_q \setminus \{\beta_1\}$.	
			Therefore, we have 
			\begin{align}
				 \nonumber &\xi= \sum^p_{i=1}s_i\sigma_i+\sum^m_{j=1} t_j\beta_j \\ 
		     	\nonumber \implies &\xi= \left(\sum^p_{i=1}s_i\sigma_i+t_1\beta_1\right)+\sum^m_{j=2} t_j\beta_j\\
		     	 \nonumber\implies &\xi= \sum_{i=1}^{p}s_i\overrightarrow{\sigma_i}^{(\W)}+ t_1 \overrightarrow{\beta_1}^{(\W)} \quad \text{(since, our hypothesis is true for $\W$)}\\
				\implies &\xi =\sum_{i=1}^p s_i \left(  \sum_{\substack{P:P \text{ is a}\\ \W \text{-trajectory,} \\ \inn(P) = \sigma_i}} w(P) \cdot \trr(P)\right) +t_1 \left(  \sum_{\substack{P:P \text{ is a}\\ \W \text{-trajectory,} \\ \inn(P) = \beta_1}} w(P) \cdot \trr(P) \right) \label{eq5}\\
				\nonumber \implies &\partial(\xi)=\sum_{i=1}^p s_i \left(  \sum_{\substack{P:P \text{ is a}\\ \W \text{-trajectory,} \\ \inn(P) = \sigma_i}} w(P) \cdot \partial(\trr(P)) \right)+t_1 \left(  \sum_{\substack{P:P \text{ is a}\\ \W \text{-trajectory,} \\ \inn(P) = \beta_1}} w(P) \cdot \partial(\trr(P)) \right)\\
				\nonumber \implies  &\langle \partial(\xi), \alpha \rangle=	\sum_{i=1}^p s_i \cdot \sum_{\substack{P:P \text{ is a}\\ \W \text{-trajectory,} \\ \inn(P) = \sigma_i}} w(P) \cdot [ \trr(P), \alpha ] +t_1\cdot  \sum_{\substack{P:P \text{ is a}\\ \W \text{-trajectory,} \\ \inn(P) = \beta_1}} w(P) \cdot [ \trr(P), \alpha ]\\
			    \nonumber \implies &\langle 0, \alpha \rangle=\sum_{i=1}^p s_i \cdot \sum_{\substack{P:P \text{ is a}\\ \W \text{-trajectory,} \\ \inn(P) = \sigma_i}} w(P) \cdot [ \trr(P), \alpha ] +t_1\cdot  \sum_{\substack{P:P \text{ is a}\\ \W \text{-trajectory,} \\ \inn(P) = \beta_1}} w(P) \cdot [ \trr(P), \alpha ]\\
				\nonumber\implies &0= \sum_{i=1}^p s_i \cdot  \sum_{\substack{P:P \text{ is a}\\ \W \text{-trajectory,} \\ \inn(P) = \sigma_i}} w(P) \cdot [ \trr(P), \alpha ] +t_1\cdot [ \beta_1, \alpha ]  ~\text{(by Observation~\ref{Ob2.13}(ii))}\\
				\nonumber\implies & t_1=-\sum_{i=1}^p s_i \cdot  \sum_{\substack{P:P \text{ is a}\\ \W \text{-trajectory,} \\ \inn(P) = \sigma_i}} w(P) \cdot [ \trr(P), \alpha ] \cdot [ \beta_1, \alpha ]\\
				\nonumber\implies & t_1=\sum_{i=1}^p s_i \cdot  \sum_{\substack{P:P \text{ is a}\\ \W \text{-trajectory,} \\ \inn(P) = \sigma_i}} w(P) \cdot(- [ \trr(P), \alpha ]\cdot [ \beta_1, \alpha ] )\\
				\implies & t_1= \sum_{i=1}^p s_i \cdot \sum_{\substack{P:P \text{ is a}\\ \V \text{-trajectory,} \\ \inn(P)= \sigma_i,\\\trr(P)=\beta_1}} w(P). \label{eq8}
			\end{align}
				Now, for any $j =2, \ldots, m$, we get	
			\begin{align}
				\nonumber t_j = &\langle \xi, \beta_j \rangle\\
				\nonumber  =& \left\langle \sum_{i=1}^p s_i \cdot \left(  \sum_{\substack{P:P \text{ is a}\\ \W \text{-trajectory,} \\ \inn(P) = \sigma_i}} w(P) \cdot \trr(P)\right) +t_1 \cdot \left(  \sum_{\substack{P:P \text{ is a}\\ \W \text{-trajectory,} \\ \inn(P) = \beta_1}} w(P) \cdot \trr(P) \right), \beta_j \right\rangle ~\text{(by Equation~\eqref{eq5})}\\
				\nonumber =&  \sum_{i=1}^p s_i \cdot \left\langle  \sum_{\substack{P:P \text{ is a}\\ \W \text{-trajectory,} \\ \inn(P) = \sigma_i}} w(P) \cdot \trr(P), \beta_j \right\rangle +t_1 \cdot \left\langle  \sum_{\substack{P:P \text{ is a}\\ \W \text{-trajectory,} \\ \inn(P) = \beta_1}} w(P) \cdot \trr(P) , \beta_j \right\rangle \\
				\nonumber =&\sum_{i=1}^p s_i \cdot\sum_{\substack{P:P \text{ is a}\\ \W \text{-trajectory} \\\inn(P)= \sigma_i,\\ \trr(P)=\beta_j}} w(P) +t_1\cdot \sum_{\substack{P:P \text{ is a}\\ \W \text{-trajectory} \\\inn(P)= \beta_1,\\ \trr(P)=\beta_j}} w(P)\\
				\nonumber =&\sum_{i=1}^p s_i \cdot \sum_{\substack{P:P \text{ is a}\\ \W \text{-trajectory} \\\inn(P)= \sigma_i,\\ \trr(P)=\beta_j}} w(P)+ \left(\sum_{i=1}^p s_i \cdot \sum_{\substack{P:P \text{ is a}\\ \V \text{-trajectory} \\ \inn(P)= \sigma_i,\\ \trr(P)=\beta_1}} w(P) \right) \sum_{\substack{P:P \text{ is a}\\ \W \text{-trajectory} \\\inn(P)= \beta_1,\\ \trr(P)=\beta_j}} w(P) ~~\text{(by Equation~\eqref{eq8})}\\
				\nonumber =&\sum_{i=1}^p s_i \cdot  \sum_{\substack{P:P \text{ is a}\\ \W \text{-trajectory} \\\inn(P)= \sigma_i,\\ \trr(P)=\beta_j}} w(P) + \sum_{i=1}^p s_i \cdot  \sum_{\substack{P:P \text{ is a}\\ \V \text{-trajectory} \\ \inn(P)= \sigma_i,\\ \trr(P)=\beta_1}} w(P) \cdot \sum_{\substack{P:P \text{ is a}\\ \V \text{-trajectory} \\\inn(P)= \beta_1,\\ \trr(P)=\beta_j}} w(P) \\
				\nonumber & ~~ (\text{by Observation~\ref{Ob2.13}(iii)})\\
				\nonumber 
				 =&\sum_{i=1}^p s_i \cdot  \sum_{\substack{P:P \text{ is a}\\ \V \text{-trajectory} \\\inn(P)= \sigma_i,\\ \trr(P)=\beta_j, \\ (\alpha,\beta)\notin P}} w(P) + \sum_{i=1}^p s_i \cdot  \sum_{\substack{P:P \text{ is a}\\ \V \text{-trajectory} \\\inn(P)= \sigma_i,\\ \trr(P)=\beta_j, \\ (\alpha,\beta)\in P}} w(P)  \\
				 =&\sum_{i=1}^p s_i \cdot  \sum_{\substack{P:P \text{ is a}\\ \V \text{-trajectory} \\\inn(P)= \sigma_i,\\ \trr(P)=\beta_j}} w(P). \label{eq9}
			\end{align}
			
			From Equation~\eqref{eq8} and Equation~\eqref{eq9}, for $1 \leq j \leq m$,  we get
			\begin{equation}
				t_j=\sum_{i=1}^p s_i \cdot  \sum_{\substack{P:P \text{ is a}\\ \V \text{-trajectory} \\\inn(P)= \sigma_i,\\ \trr(P)=\beta_j}} w(P). \label{eq10}
			\end{equation}
		
			Now, we have
			\begin{align*}
				\xi =& \sum^p_{i=1}s_i\sigma_i+\sum^m_{j=1} t_j\beta_j\\
				&= \sum_{i=1}^p s_i\sigma_i + \sum_{j=1}^m \left( \sum_{i=1}^p s_i \cdot  \sum_{\substack{P:P \text{ is a}\\ \V \text{-trajectory} \\\inn(P)= \sigma_i,\\ \trr(P)=\beta_j}} w(P) \right) \beta_j ~~\text{(from Equation~\eqref{eq10})}\\
				&= \sum_{i=1}^p s_i\sigma_i + \sum_{i=1}^p s_i  \left(  \sum_{j=1}^m  \sum_{\substack{P:P \text{ is a}\\ \V \text{-trajectory} \\\inn(P)= \sigma_i,\\ \trr(P)=\beta_j}} w(P) \cdot \beta_j  \right)  \\ 
				&= \sum_{i=1}^p s_i \left( \sigma_i + \sum_{j=1}^m \sum_{\substack{P:P \text{ is a}\\ \V \text{-trajectory} \\\inn(P)= \sigma_i,\\ \trr(P)=\beta_j}} w(P) \cdot \beta_j  \right)  \\ 
				&=\sum_{i=1}^p s_i  \cdot \sum_{\substack{P:P \text{ is a}\\ \V \text{-trajectory} \\ \text{from } \sigma_i}} w(P) \cdot \trr(P)\\
				&=  \sum_{i=1}^p s_i \overrightarrow{\sigma_i}^{(\V)}.
			\end{align*}
			
			\textbf{Case 3.}
			If $q=k-1$, $\crit^{(\W)}_q=\crit^{(\V)}_q \cup \{\alpha\}$, $U^{(\W)}_q=U^{(\V)}_q \setminus \{\alpha\}$, and $D^{(\W)}_q=D^{(\V)}_q$ .  Since, our hypothesis is true for $\W$, we get 
			\begin{align*}
				\xi &= \sum^p_{i=1}s_i\sigma_i+\sum^m_{j=1} t_j\beta_j\\ 
				&=  \sum^p_{i=1}s_i\sigma_i + 0\cdot \alpha +\sum^m_{j=1} t_j\beta_j \label{eq4}\\ 
				&=\sum_{i=1}^p s_i\overrightarrow{\sigma_i}^{(\W)}+0\cdot \overrightarrow{\alpha}^{(\W)}\\
				&=\sum_{i=1}^p s_i\overrightarrow{\sigma_i}^{(\W)}.
			\end{align*}

			Note that, for any $i\in \{1,\ldots,p\}$, the $\V$-trajectories, starting from any $\sigma_i$, do not contain the pair $(\alpha, \beta)$. So, the $\V$-trajectories and $\W$-trajectories, which start from $\sigma_i$, are same. Which implies, for all $i\in \{1,\ldots,p\}$, $\overrightarrow{\sigma_i}^{(\V)}=\overrightarrow{\sigma_i}^{(\W)}$. Therefore, we get
			$\xi=\sum_{i=1}^p s_i\overrightarrow{\sigma_i}^{(\W)}=\sum_{i=1}^p s_i\overrightarrow{\sigma_i}^{(\V)}$.
			
		\end{proof}
    \begin{lemma} \label{l3.3}
    	Let $\K$ be a $d$-dimensional simplicial complex and $\V$ be a gradient vector field on it. For some  $q \in \{0,\ldots,d-1\}$, let  $\crit^{(\V)}_{q}(\K)=\{\sigma_1, \ldots, \sigma_p\}$. Then, for any $\tau\in \crit_{q+1}^{(\V)}(\K)$, we have the following.
    	\[\partial(\overrightarrow{\tau}^{(\V)})=\sum_{i=1}^{p} \left(\sum_{\substack{P:P \text{ is a}\\ \V \text{-trajectory,}\\ \inn(P)=\tau}}w(P)\cdot [ \trr(P),\sigma_i ] \right) \overrightarrow{\sigma_i}^{(\V)}.\]
    \end{lemma}
    \begin{proof}
    Let $U_{q}^{(\V)}=\{\alpha_1,\ldots, \alpha_l\}$ and $D^{(\V)}_{q}=\{\beta_1, \ldots, \beta_m\}$. Recall that, $S_q= \crit_{q+1}^{(\V)} \sqcup U_{q}^{(\V)}  \sqcup D^{(\V)}_{q}$ is the standard basis of $C_q(\K,\mathbb{Z}$). We have
    \begin{align}
    \nonumber \partial(\overrightarrow{\tau}^{(\V)})=& \partial \left( \sum_{\substack{P:P \text{ is a}\\ \V \text{-trajectory,}\\ \inn(P)=\tau}}w(P)\cdot \trr(P)\right)\\
    \nonumber =& \sum_{\substack{P:P \text{ is a}\\ \V \text{-trajectory,}\\ \inn(P)=\tau}}w(P)\cdot \partial(\trr(P))\\
     \nonumber=& \sum_{\substack{P:P \text{ is a}\\ \V \text{-trajectory,}\\ \inn(P)=\tau}}w(P)\cdot \left( \sum_{\sigma \in S_{q}}[\trr(P),\sigma]\cdot \sigma\right)\\
     \nonumber=&  \sum_{\sigma \in S_{q}} \left( \sum_{\substack{P:P \text{ is a}\\ \V \text{-trajectory,}\\ \inn(P)=\tau}}w(P)\cdot[\trr(P),\sigma]\right) \sigma \\
    \nonumber=& \sum_{i=1}^{p} \left(\sum_{\substack{P:P \text{ is a}\\ \V \text{-trajectory,}\\ \inn(P)=\tau}}w(P)\cdot [ \trr(P),\sigma_i ] \right) \sigma_i+  \sum_{j=1}^{l}\left( \sum_{\substack{P:P \text{ is a}\\ \V \text{-trajectory,}\\ \inn(P)=\tau}}w(P)\cdot [ \trr(P),\alpha_j ] \right)\alpha_j\\
    & + \sum_{k=1}^{m} \left( \sum_{\substack{P:P \text{ is a}\\ \V \text{-trajectory,}\\ \inn(P)=\tau}}w(P)\cdot [ \trr(P),\beta_k ]\right)  \beta_k. \label{eq11}
    \end{align}	
     We note that, the Equation~\eqref{eq11} is the representation $\partial(\overrightarrow{\tau}^{(\V)})$ with respect to the basis $S_q$. From the definition of the weight of a trajectory and the acyclicity of $\V$, for any $\alpha_j \in U_q^{(\V)}$, we have the following.
    \begin{align}
    	\nonumber \sum_{\substack{P:P \text{ is a } \\ \V\text{-trajectory,} \\ \inn(P)=\tau, \\ \trr(P) =\V(\alpha_j)}} w(P)&= \sum_{\substack{P:P \text{ is a } \\ \V\text{-trajectory,} \\ \inn(P)=\tau, \\ \trr(P)\neq \V(\alpha_j)}} w(P) \cdot (-[ \trr(P),\alpha_j ]\cdot[ \V(\alpha_j),\alpha_j ])\\
    	& = - \sum_{\substack{P:P \text{ is a } \\ \V\text{-trajectory,} \\ \inn(P)=\tau, \\ \trr(P)\neq \V(\alpha_j)}} w(P) \cdot [ \trr(P),\alpha_j ]\cdot[ \V(\alpha_j),\alpha_j ]. \label{eq12}
    \end{align}
    Now, for the coefficient of $\alpha_j$ in Equation~\eqref{eq11}, we have
    \begin{align} 
        \nonumber &\sum_{\substack{P:P \text{ is a } \\ \V\text{-trajectory,} \\ \inn(P)=\tau}} w(P) \cdot [ \trr(P),\alpha_j ]\\
        \nonumber =& \sum_{\substack{P:P \text{ is a } \\ \V\text{-trajectory,} \\ \inn(P)=\tau, \\ \trr(P)\neq\V(\alpha_j)}} w(P) \cdot [ \trr(P),\alpha_j ] + \sum_{\substack{P:P \text{ is a } \\ \V\text{-trajectory,} \\ \inn(P)=\tau, \\ \trr(P)=\V(\alpha_j)}} w(P) \cdot [ \trr(P),\alpha_j ] \\
    	 \nonumber=& \sum_{\substack{P:P \text{ is a } \\ \V\text{-trajectory,} \\ \inn(P)=\tau, \\ \trr(P)\neq \V(\alpha_j)}} w(P)\cdot [ \trr(P),\alpha_j ] + \sum_{\substack{P:P \text{ is a } \\ \V\text{-trajectory,} \\ \inn(P)=\tau, \\ \trr(P)=\V(\alpha_j)}} w(P) \cdot [ \V(\alpha_j),\alpha_j ]\\
    	\nonumber=& \sum_{\substack{P:P \text{ is a } \\ \V\text{-trajectory,} \\ \inn(P)=\tau, \\ \trr(P)\neq\V(\alpha_j)}} w(P) \cdot [ \trr(P),\alpha_j ] \\
    	\nonumber & + \left( -\sum_{\substack{P:P \text{ is a } \\ \V\text{-trajectory,} \\ \inn(P)=\tau, \\ \trr(P)\neq\V(\alpha_j)}} w(P) \cdot [ \trr(P),\alpha_j ]\cdot[ \V(\alpha_j),\alpha_j ]\right) \cdot [ \V(\alpha_j),\alpha_j ] \\
    	\nonumber & \text{(from Equation~\eqref{eq12})}\\
    	\nonumber =& \sum_{\substack{P:P \text{ is a } \\ \V\text{-trajectory,} \\ \inn(P)=\tau, \\ \trr(P)\neq \V(\alpha_j)}} w(P) [ \trr(P),\alpha_j ] -\sum_{\substack{P:P \text{ is a } \\ \V\text{-trajectory,} \\ \inn(P)=\tau, \\ \trr(P)\neq \V(\alpha_j)}} w(P) [ \trr(P),\alpha_j ]\\
        \nonumber= &~ 0.
    \end{align}
    Putting this value in Equation~\eqref{eq11}, we get
       	\[\partial(\overrightarrow{\tau}^{(\V)})=\sum_{i=1}^{p} \left( \sum_{\substack{P:P \text{ is a}\\ \V \text{-trajectory,}\\ \inn(P)=\tau}}w(P)\cdot [ \trr(P),\sigma_i ] \right) \sigma_i+\sum_{k=1}^{m} \left( \sum_{\substack{P:P \text{ is a}\\ \V \text{-trajectory,}\\ \inn(P)=\tau}}w(P)\cdot [ \trr(P),\beta_k ] \right) \beta_k.\]
       	Now since, $\partial(\overrightarrow{\tau}^{(\V)}) \in Z_q(\K)$, from Lemma~\ref{l3.2}, we have
       	\[\partial(\overrightarrow{\tau}^{(\V)})=\sum_{i=1}^{p} \left( \sum_{\substack{P:P \text{ is a}\\ \V \text{-trajectory,}\\ \inn(P)=\tau}}w(P)\cdot [ \trr(P),\sigma_i ]\right) \overrightarrow{\sigma_i}^{(\V)}.\]
    \end{proof}
    \begin{lemma} \label{l-3.4}
    	Let $\K$ be a $d$-dimensional simplicial complex and $\V$ be a gradient vector field on it. Let $\sigma\in \crit_q^{(\V)}(\K)$. Then, for any $\xi \in C_q(\K, \mathbb{Z})$, 
    	\[\langle \xi,\overrightarrow{\sigma}^{(\V)} \rangle_{\tilde{S}_q^{(\V)}}=\langle \overleftarrow{\sigma}^{(\V)},\xi\rangle_{S_q}.\]
    	
    \end{lemma}
    \begin{proof}
    	We will prove this lemma by induction on the cardinality of the gradient vector field $\V$. When $|\V|=0$, $\widetilde{S}_q^{(\V)}=S_q$ and $\overrightarrow{\sigma}^{(\V)}=\overleftarrow{\sigma}^{(\V)}=\sigma$. Hence, our hypothesis is true.
    	
    	Let our hypothesis is true for the gradient vector fields with cardinality less than $n$, and let $|\V|=n$. We recall that, $S_q= \crit_q^{(\V)} \sqcup U_q^{(\V)} \sqcup D_q^{\V}$ and $\widetilde{S}_q^{(\V)}= \overrightarrow{\crit}_q^{(\V)} \sqcup \widetilde{U}_q^{(\V)} \sqcup D_q^{\V}$.

    	\textbf{Case 1:} In this case we assume that $U_q^{(\V)}\neq \emptyset$. Let, $\crit_q^{(\V)}=\{\sigma_1, \ldots, \sigma_p\}$, $U_{q}^{(\V)}=\{\alpha_1,\ldots,\alpha_l\}$ and $D_q^{(\V)}=\{\beta_1,\ldots,\beta_m\}$.
    	
    	 Let $\V(\alpha_j^{(q)})=\eta_j^{(q+1)}$, for all $j\in \{1,\ldots, l\}$. We define $\W=\V \setminus (\alpha_1^{(q)},\eta_1^{(q+1)})$. We note that, $\W$ is also a gradient vector field on $\K$, with $|\W|=n-1$. We have,
         \[\widetilde{S}_q^{(\W)}=\{\overrightarrow{\sigma_1}^{(\W)},\ldots, \overrightarrow{\sigma_p}^{(\W)}, \overrightarrow{\alpha_1}^{(\W)}, \widetilde{\alpha}_2, \ldots, \widetilde{\alpha}_l,, \beta_1, \ldots, \beta_m\}.\]
    	With respect to the basis $\widetilde{S}_q^{(\W)}$, let $\xi$ have the following expression   
    	\begin{align} \label{xi}
    		  \xi= & \sum_{i=1}^{p} s_i\cdot \overrightarrow{\sigma_i}^{(\W)} + t_1 \cdot\overrightarrow{\alpha_1}^{(\W)}  + \sum_{j=2}^{l} t_j\cdot \widetilde{\alpha}_j + \sum_{k=1}^{m}r_k \cdot\beta_k.
    	\end{align}
    	Here, $s_i,t_j,r_k \in \mathbb{Z}$. Now, since our hypothesis is true for $\W$, we have
    	\begin{equation} \label{hyp}
    		s_i(=\langle \xi, \overrightarrow{\sigma_i}^{(\W)}\rangle_{\widetilde{S}_q^{(\W)}})= \langle \overleftarrow{\sigma_i}^{(\W)},\xi\rangle_{S_q} \text{ and } t_1(=\langle \xi,  \overrightarrow{\alpha_1}^{(\W)}\rangle_{\widetilde{S}_q^{(\W)}})= \langle \overleftarrow{\alpha_1}^{(\W)},\xi\rangle_{S_q}.
    	\end{equation}
    	
    	From Lemma~\ref{l3.3} , we get
    	\begin{align*}
    		\partial(\overrightarrow{\eta_1}^{(\W)})&=\sum_{i=1}^{p} \sum_{\substack{P:P \text{ is a}\\ \W \text{-trajectory,}\\ \inn(P)=\eta_1}}w(P)\cdot [ \trr(P),\sigma_i ]  \cdot \overrightarrow{\sigma_i}^{(\W)} + \sum_{\substack{P:P \text{ is a}\\ \W \text{-trajectory,}\\ \inn(P)=\eta_1}}w(P)\cdot [ \trr(P),\alpha_1 ]  \cdot \overrightarrow{\alpha_1}^{(\W)}\\
    		\implies  \partial(\overrightarrow{\eta_1}^{(\W)})&=\sum_{i=1}^{p} \sum_{\substack{P:P \text{ is a}\\ \W \text{-trajectory,}\\ \inn(P)=\eta_1}}w(P)\cdot [ \trr(P),\sigma_i ]  \cdot \overrightarrow{\sigma_i}^{(\W)} + [ \eta_1,\alpha_1 ] \cdot \overrightarrow{\alpha_1}^{(\W)}~ \text{ (by Observation~\ref{Ob2.13}(ii))}\\
    		\implies~~~~ \overrightarrow{\alpha_1}^{(\W)}&= [ \eta_1,\alpha_1 ] \cdot\partial(\overrightarrow{\eta_1}^{(\W)})- [ \eta_1,\alpha_1 ] \cdot\sum_{i=1}^{p} \sum_{\substack{P:P \text{ is a}\\ \W \text{-trajectory,}\\ \inn(P)=\eta_1}}w(P)\cdot [ \trr(P),\sigma_i ]  \cdot \overrightarrow{\sigma_i}^{(\W)}\\
    		&= [ \eta_1,\alpha_1 ] \cdot\partial(\overrightarrow{\eta_1}^{(\W)})+\sum_{i=1}^{p} \sum_{\substack{P:P \text{ is a}\\ \W \text{-trajectory,}\\ \inn(P)=\eta_1}}(- [ \eta_1,\alpha_1 ] \cdot w(P)\cdot [ \trr(P),\sigma_i ])  \cdot \overrightarrow{\sigma_i}^{(\W)}\\
    		&= [ \eta_1,\alpha_1 ] \cdot\partial(\overrightarrow{\eta_1}^{(\W)})+\sum_{i=1}^{p} \sum_{\substack{Q:Q \text{ is a}\\ \text{co-}\V \text{-trajectory,}\\ \inn(Q)=\sigma_i , \\ \trr(Q)=\alpha_1}}w(Q)  \cdot \overrightarrow{\sigma_i}^{(\W)}\\
    		&= [ \eta_1,\alpha_1 ] \cdot\partial (\sum_{\substack{P:P \text{ is a }\\ \W \text{-trajectory,}\\ \inn(P)=\eta_1}}w(P)\cdot \trr(P))+\sum_{i=1}^{p} \sum_{\substack{Q:Q \text{ is a}\\ \text{co-}\V \text{-trajectory,}\\ \inn(Q)=\sigma_i , \\ \trr(Q)=\alpha_1}}w(Q)  \cdot \overrightarrow{\sigma_i}^{(\W)}\\
    		&= [ \eta_1,\alpha_1 ] \cdot\partial (\sum_{j=1}^{l}  \sum_{\substack{P:P \text{ is a }\\ \W \text{-trajectory,}\\ \inn(P)=\eta_1 ,\\ \trr(P)=\eta_j }}w(P)\cdot \eta_j)+\sum_{i=1}^{p} \sum_{\substack{Q:Q \text{ is a}\\ \text{co-}\V \text{-trajectory,}\\ \inn(Q)=\sigma_i , \\ \trr(Q)=\alpha_1}}w(Q)  \cdot \overrightarrow{\sigma_i}^{(\W)} \\
    		&= [ \eta_1,\alpha_1 ] \cdot \sum_{j=1}^{l}  \sum_{\substack{P:P \text{ is a }\\ \W \text{-trajectory,}\\ \inn(P)=\eta_1 ,\\ \trr(P)=\eta_j}}w(P)\cdot \partial(\eta_j)+\sum_{i=1}^{p} \sum_{\substack{Q:Q \text{ is a}\\ \text{co-}\V \text{-trajectory,}\\ \inn(Q)=\sigma_i , \\ \trr(Q)=\alpha_1}}w(Q)  \cdot \overrightarrow{\sigma_i}^{(\W)}\\
    		&= [ \eta_1,\alpha_1 ] \cdot \sum_{j=1}^{l}  \sum_{\substack{P:P \text{ is a }\\ \W \text{-trajectory,}\\ \inn(P)=\eta_1 ,\\ \trr(P)=\eta_j}}w(P)\cdot \widetilde{\alpha}_j+\sum_{i=1}^{p} \sum_{\substack{Q:Q \text{ is a}\\ \text{co-}\V \text{-trajectory,}\\ \inn(Q)=\sigma_i , \\ \trr(Q)=\alpha_1}}w(Q)  \cdot \overrightarrow{\sigma_i}^{(\W)}\\
    		&=  \sum_{j=1}^{l}  [ \eta_1,\alpha_1 ] \cdot \sum_{\substack{P:P \text{ is a }\\ \W \text{-trajectory,}\\ \inn(P)=\eta_1,\\ \trr(P)=\eta_j}}w(P)\cdot \widetilde{\alpha}_j+\sum_{i=1}^{p} \sum_{\substack{Q:Q \text{ is a}\\ \text{co-}\V \text{-trajectory,}\\ \inn(Q)=\sigma_i , \\ \trr(Q)=\alpha_1}}w(Q)  \cdot \overrightarrow{\sigma_i}^{(\W)}.
    	    \end{align*}
    	
    	By putting this value of $\overrightarrow{\alpha}_1^{(\W)}$ in Equation~\eqref{xi}, we get 
    	
    	\begin{align} \label{nxi}
    	\nonumber \xi= & \sum_{i=1}^{p} s_i \cdot \overrightarrow{\sigma_i}^{(\W)} + t_1 \cdot \left(  \sum_{j=1}^{l}  [ \eta_1,\alpha_1 ] \cdot \sum_{\substack{P:P \text{ is a }\\ \W \text{-trajectory,}\\ \inn(P)=\eta_1 ,\\ \trr(P)=\eta_j}}w(P)\cdot \widetilde{\alpha}_j  
        +  \sum_{i=1}^{p} \sum_{\substack{Q:Q \text{ is a}\\ \text{co-}\V \text{-trajectory,}\\ \inn(Q)=\sigma_i ,\\ \trr(Q)=\alpha_1}}w(Q)  \cdot \overrightarrow{\sigma_i}^{(\W)} \right)  \\
        \nonumber & + \sum_{j=2}^{l} t_j\cdot \widetilde{\alpha}_j + \sum_{k=1}^{m}r_k \cdot\beta_k\\
        \nonumber = & \sum_{i=1}^{p} s_i \cdot \overrightarrow{\sigma_i}^{(\W)} + \sum_{j=1}^{l} t_1 \cdot  [ \eta_1,\alpha_1 ] \cdot \sum_{\substack{P:P \text{ is a }\\ \W \text{-trajectory,}\\ \inn(P)=\eta_1 ,\\ \trr(P)=\eta_j}}w(P)\cdot \widetilde{\alpha}_j  
        +  \sum_{i=1}^{p} t_1 \cdot \sum_{\substack{Q:Q \text{ is a}\\ \text{co-}\V \text{-trajectory,}\\ \inn(Q)=\sigma_i ,\\ \trr(Q)=\alpha_1}}w(Q)  \cdot \overrightarrow{\sigma_i}^{(\W)}  \\
        \nonumber & + \sum_{j=2}^{l} t_j\cdot \widetilde{\alpha}_j + \sum_{k=1}^{m}r_k \cdot\beta_k\\
        \nonumber = & \sum_{i=1}^{p} \left( s_i +  t_1 \cdot \sum_{\substack{Q:Q \text{ is a}\\ \text{co-}\V \text{-trajectory,}\\ \inn(Q)=\sigma_i ,\\ \trr(Q)=\alpha_1}}w(Q) \right)\cdot \overrightarrow{\sigma_i}^{(\W)} + t_1 \cdot  [ \eta_1,\alpha_1 ] \cdot \widetilde{\alpha}_1 \\
        \nonumber & + \sum_{j=2}^{l} \left( t_j + t_1 \cdot  [ \eta_1,\alpha_1 ] \cdot \sum_{\substack{P:P \text{ is a }\\ \W \text{-trajectory,}\\ \inn(P)=\eta_1 ,\\ \trr(P)=\eta_j}}w(P) \right) \cdot \widetilde{\alpha}_j + \sum_{k=1}^{m}r_k \cdot\beta_k.\\
    	\end{align}
        Note that the Equation~\eqref{nxi} is the representation of $\xi$, with respect to the basis $\widetilde{S_q}^{(\V)}$. So, for any $i\in \{1,\ldots, p\}$, the coefficient of $\overrightarrow{\sigma_i}^{(\V)}$ in $\xi$ with respect to the basis $\widetilde{S}_q^{(\V)}$, 
    	\begin{align*} 
    	&\langle \xi, \overrightarrow{\sigma_i}^{(\V)}\rangle_{\widetilde{S}_q^{(\V)}} \\
    	= & s_i +  t_1 \cdot \sum_{\substack{Q:Q \text{ is a}\\ \text{co-}\V \text{-trajectory,}\\ \inn(Q)=\sigma_i ,\\ \trr(Q)=\alpha_1}}w(Q)\\
    	=&\left\langle\overleftarrow{\sigma_i}^{(\W)},\xi\right\rangle+ \left\langle\overleftarrow{\alpha_1}^{(\W)},\xi\right\rangle \cdot \sum_{\substack{Q:Q \text{ is a}\\ \text{co-}\V \text{-trajectory,}\\ \inn(Q)=\sigma_i ,\\ \trr(Q)=\alpha_1}}w(Q)~~\text{(from Equation~\eqref{hyp})}\\
    	=&\left\langle\sum_{\substack{Q:Q \text{ is a}\\ \text{co-}\W \text{-trajectory,}\\ \inn(Q)=\sigma_i}}w(Q)\cdot \trr(Q),~\xi\right\rangle+ \left\langle\sum_{\substack{Q:Q \text{ is a}\\ \text{co-}\W \text{-trajectory,}\\ \inn(Q)=\alpha_1}}w(Q)\cdot \trr(Q),~\xi\right\rangle \cdot \sum_{\substack{Q:Q \text{ is a}\\ \text{co-}\V \text{-trajectory,}\\ \inn(Q)=\sigma_i ,\\ \trr(Q)=\alpha_1}}w(Q)\\
    	=&\left\langle\sum_{\substack{Q:Q \text{ is a}\\ \text{co-}\W \text{-trajectory,}\\ \inn(Q)=\sigma_i}}w(Q)\cdot \trr(Q),~\xi\right\rangle+ \left\langle \sum_{\substack{Q:Q \text{ is a}\\ \text{co-}\V \text{-trajectory,}\\ \inn(Q)=\sigma_i ,\\ \trr(Q)=\alpha_1}}w(Q) \cdot \sum_{\substack{Q:Q \text{ is a}\\ \text{co-}\W \text{-trajectory,}\\ \inn(Q)=\alpha_1}}w(Q)\cdot \trr(Q),~ \xi\right\rangle\\
    	=&\left\langle\sum_{\substack{Q:Q \text{ is a}\\ \text{co-}\V \text{-trajectory,}\\ \inn(Q)=\sigma_i ,\\ (\alpha_1,\eta_1)\notin Q}}w(Q)\cdot \trr(Q),~\xi\right\rangle+ \left\langle\sum_{\substack{Q:Q \text{ is a}\\ \text{co-}\V \text{-trajectory,}\\ \inn(Q)=\sigma_i ,\\ (\alpha_1,\eta_1)\in Q}}w(Q)\cdot \trr(Q),~\xi\right\rangle\\
    	=&\left\langle\sum_{\substack{Q:Q \text{ is a}\\ \text{co-}\V \text{-trajectory,}\\ \inn(Q)=\sigma_i ,\\ (\alpha_1,\eta_1)\notin Q}}w(Q)\cdot \trr(Q)+\sum_{\substack{Q:Q \text{ is a}\\ \text{co-}\V \text{-trajectory,}\\ \inn(Q)=\sigma_i ,\\ (\alpha_1,\eta_1)\in Q}}w(Q)\cdot \trr(Q),~\xi\right\rangle\\
    	=&\left\langle\sum_{\substack{Q:Q \text{ is a}\\ \text{co-}\V \text{-trajectory,}\\ \inn(Q)=\sigma_i}}w(Q)\cdot \trr(Q),~\xi\right\rangle\\
    	=&\langle\overleftarrow{\sigma_i}^{(\V)},~\xi\rangle_{S_q}.
    	\end{align*}
    	
    		\textbf{Case 2:}
    	When $U_q^{(\V)}=\emptyset$, we have
    	\begin{align} 
    		\nonumber \xi &= \sum_{i=1}^p \langle \xi,\overrightarrow{\sigma_i}^{(\V)} \rangle_{\widetilde{S}_q^{(\V)}} \cdot \overrightarrow{\sigma_i}^{(\V)}+ \sum_{k=1}^{m} \langle \xi, \beta_k \rangle_{\widetilde{S}_q^{(\V)}} \cdot \beta_k \\
        	\nonumber &= \sum_{i=1}^p \langle \xi, \overrightarrow{\sigma_i}^{(\V)} \rangle_{\widetilde{S}_q^{(\V)}}  \left( \sum_{\substack{P:P \text{ is a}\\ \V \text{-trajectory,}\\ \inn(P)=\sigma_i}}w(P) \cdot \trr(P) \right)+ \sum_{k=1}^{m} \langle \xi, \beta_k \rangle_{\widetilde{S}_q^{(\V)}} \cdot\beta_k \\
    		\nonumber  &= \sum_{i=1}^p \langle \xi,  \overrightarrow{\sigma_i}^{(\V)} \rangle_{\widetilde{S}_q^{(\V)}}  \left( \sigma_i + \sum_{k=1}^{m}\sum_{\substack{P:P \text{ is a}\\ \V \text{-trajectory,}\\ \inn(P)=\sigma_i, \\ \trr(P)=\beta_k}}w(P) \cdot \beta_k \right) + \sum_{k=1}^{m} \langle \xi, \beta_k \rangle_{\widetilde{S}_q^{(\V)}} \cdot\beta_k \\
    		\nonumber  &= \sum_{i=1}^p \langle \xi,  \overrightarrow{\sigma_i}^{(\V)} \rangle_{\widetilde{S}_q^{(\V)}}  \cdot \sigma_i +  \sum_{i=1}^p \langle \xi,  \overrightarrow{\sigma_i}^{(\V)} \rangle_{\widetilde{S}_q^{(\V)}} \cdot \sum_{k=1}^{m}\sum_{\substack{P:P \text{ is a}\\ \V \text{-trajectory,}\\ \inn(P)=\sigma_i, \\ \trr(P)=\beta_k}}w(P) \cdot \beta_k + \sum_{k=1}^{m} \langle \xi, \beta_k \rangle_{\widetilde{S}_q^{(\V)}} \cdot\beta_k \\
    		&= \sum_{i=1}^p \left( \langle \xi, \overrightarrow{\sigma_i}^{(\V)} \rangle_{\widetilde{S}_q^{(\V)}} \right) \sigma_i  + \sum_{k=1}^{m}\left( \sum_{i=1}^p \langle \xi,  \overrightarrow{\sigma_i}^{(\V)} \rangle_{\widetilde{S}_q^{(\V)}} \cdot \sum_{\substack{P:P \text{ is a}\\ \V \text{-trajectory,}\\ \inn(P)=\sigma_i \\ \trr(P)=\beta_k}}w(P)+\langle \xi, \beta_k \rangle_{\widetilde{S}_q^{(\V)}} \right)\beta_k. \label{e17}
    	\end{align}
        Equation~\eqref{e17} is the representation of $\xi$ with respect to the basis $S_q$, and hence, for $i \in \{1,\ldots, p\}$, we have
    	\begin{align*}
    		&\langle \xi, \sigma_i\rangle_{S_q}=\langle \xi, \overrightarrow{\sigma_i}^{(\V)}\rangle_{\widetilde{S}_q^{(\V)}}\\
    	   	\implies & \langle \xi, \overleftarrow{\sigma_i}^{(\V)}\rangle_{S_q}=\langle \xi, \overrightarrow{\sigma_i}^{(\V)}\rangle_{\widetilde{S}_q^{(\V)}}~ \text{(when $U_q^{(\V)}=\emptyset$, $\overleftarrow{\sigma_i}^{(\V)}=\sigma_i$)}\\
    	   	\implies & \langle \overleftarrow{\sigma_i}^{(\V)}, \xi \rangle_{S_q}=\langle \xi, \overrightarrow{\sigma_i}^{(\V)}\rangle_{\widetilde{S}_q^{(\V)}}.	
    	\end{align*}

    \end{proof}
   
	\begin{lemma} \label{3.4.}
		Let $\K$ be a $d$-dimensional simplicial complex with a gradient vector field $\V$ on it. Let $\phi_\#: C_\#(\K, \mathbb{Z}) \rightarrow C_\#(\K,\mathbb{Z})$ be a chain map. Then for any $\beta \in D_q^{(\V)}$, $1 \leq q \leq d$,
		\[\langle\phi_{q}(\beta),\beta \rangle_{\widetilde{S}_q^{(\V)}}=\langle \phi_{q-1}(\partial(\beta)),\partial(\beta) \rangle_{\widetilde{S}_{q-1}^{(\V)}}.\]
	\end{lemma}
	\begin{proof}
		Let  $\overrightarrow{\crit}_q^{(\V)}(\K)= \{\overrightarrow{\sigma_1}^{(\V)}, \ldots , \overrightarrow{\sigma_p}^{(\V)}\}$, $\widetilde{U}_q^{(\V)}= \{\widetilde{\alpha}_1,\ldots, \widetilde{\alpha}_l\}$ and $D_q^{(\V)}= \{\beta_1,\ldots,\beta_m\}$. Let
		\[\phi_q(\beta)= \sum_{i=1}^p s_i\overrightarrow{\sigma}_i^{(\V)} + \sum_{j=1}^l t_j \widetilde{\alpha}_j^{(\V)} + \sum_{k=1}^m r_k\beta_k,\]
		Without loss of generality, let us assume $\beta=\beta_1$. Then, we have $\langle \phi_{q}(\beta),\beta \rangle_{\widetilde{S}_q^{(\V)}}= r_1$. 
		
			\begin{figure}[h]
			\centering
			\[\begin{tikzcd}
				\cdots && {C_{q+1}(\K,\mathbb{Z})} && {C_q(\K, \mathbb{Z})} && {C_{q-1}(\K, \mathbb{Z})} && \cdots \\
				\\
				\cdots && {C_{q+1}(\K, \mathbb{Z})} && {C_{q}(\K, \mathbb{Z})} && {C_{q-1}(\K, \mathbb{Z})} && \cdots
				\arrow["{\partial_{q+2}}", from=1-1, to=1-3]
				\arrow["{\partial_{q+1}}", from=1-3, to=1-5]
				\arrow["{\phi_{q+1}}", from=1-3, to=3-3]
				\arrow["{\partial_q}", from=1-5, to=1-7]
				\arrow["{\phi_q}", from=1-5, to=3-5]
				\arrow["{\partial_{q-1}}", from=1-7, to=1-9]
				\arrow["{\phi_{q-1}}", from=1-7, to=3-7]
				\arrow["{\partial_{q+2}}", from=3-1, to=3-3]
				\arrow["{\partial_{q+1}}", from=3-3, to=3-5]
				\arrow["{\partial_{q}}", from=3-5, to=3-7]
				\arrow["{\partial_{q-1}}", from=3-7, to=3-9]
			\end{tikzcd}\]
			\caption{Commutative diagram corresponding to the chain map $\phi_\#:C_\#(\K, \mathbb{Z})\rightarrow C_\#(\K, \mathbb{Z})$.}\label{f1}
			\label{fig:enter-label}
		\end{figure}

		By the commutativity of the diagram (Figure~\ref{f1}), we get
		\begin{align*}
			&\phi_{q-1} (\partial (\beta))=\partial(\phi_{q}(\beta))\\
			\implies  &\phi_{q-1} (\partial (\beta))=\partial(\alpha)\\
			\implies  &\phi_{q-1} (\partial (\beta))=\partial( \sum_{i=1}^p s_i\overrightarrow{\sigma}_i^{(\V)} + \sum_{j=1}^l t_j \widetilde{\alpha}_j^{(\V)} + \sum_{k=1}^m r_k\beta_k)\\
			\implies  &\phi_{q-1} (\partial (\beta))= \sum_{i=1}^p s_i\partial(\overrightarrow{\sigma}_i^{(\V)}) + \sum_{j=1}^l t_j \partial(\widetilde{\alpha}_j^{(\V)}) + \sum_{k=1}^m r_k\partial(\beta_k)\\
			\implies  &\phi_{q-1} (\partial (\beta))=\sum_{i=1}^p s_i \partial(\overrightarrow{\sigma}_i^{(\V)}) + \sum_{k=1}^m r_k \partial(\beta_k)\\
			\implies  & \langle \phi_{q-1} (\partial (\beta)),\partial(\beta)\rangle_{\widetilde{S}_q^{(\V)}}= r_1.
		\end{align*}
		Hence, $\langle \phi_{q}(\beta),\beta \rangle_{\widetilde{S}_q^{(\V)}}=\langle \phi_{q-1} (\partial (\beta)),\partial(\beta)\rangle_{\widetilde{S}_q^{(\V)}}$.
	\end{proof}

    Now we proceed to prove the combinatorial Hopf trace formula.
   
	\begin{proof}[Proof of Theorem~\ref{t1.10}]
		By Lemma~\ref{l-basis}, for each $0\leq q \leq d$, $\widetilde{S}_q^{(\V)}$ is a basis of $C_q(\K, \mathbb{Z})$. Expanding the left-hand side of the Equation~\eqref{eq-main}, with respect to the basis $\widetilde{S}_q^{(\V)}$, and noting that the trace is independent of the choice of basis, we get
		\begin{align*}
			&\sum_{q=0}^d (-1)^q \operatorname{tr}(\phi_q)\\
			= & \sum_{q=0}^d (-1)^q  \left( \sum_{\overrightarrow{\sigma}\in\overrightarrow{\crit}_q^{(\V)}(\K)} \langle \phi_q(\overrightarrow{\sigma}^{(\V)}),\overrightarrow{\sigma}^{(\V)}\rangle_{\widetilde{S}_q^{(\V)}} + \sum_{\widetilde{\alpha}\in \widetilde{U}_q^{(\V)}} \langle \phi^*_q(\widetilde{\alpha}),\widetilde{\alpha}\rangle_{\widetilde{S}_q^{(\V)}} + \sum_{\beta\in D_q^{(\V)}} \langle  \phi^*_q(\beta), \beta \rangle_{\widetilde{S}_q^{(\V)}} \right) \\
		\end{align*}
		\begin{align*}
			= & \sum_{q=0}^d (-1)^q  \sum_{\sigma\in\crit_q^{(\V)}(\K)} \langle \phi_q(\overrightarrow{\sigma}^{(\V)}),\overrightarrow{\sigma}^{(\V)}\rangle_{\widetilde{S}_q^{(\V)}} + \sum_{q=0}^d (-1)^q  \sum_{\alpha\in U_q^{(\V)}} \langle \phi^*_q(\widetilde{\alpha}),\widetilde{\alpha}\rangle_{\widetilde{S}_q^{(\V)}}\\ &+  \sum_{q=0}^d (-1)^q \sum_{\beta\in D_q^{(\V)}} \langle  \phi^*_q(\beta), \beta \rangle_{\widetilde{S}_q^{(\V)}}\\
			= & \sum_{q=0}^d (-1)^q  \sum_{\sigma\in\crit_q^{(\V)}(\K)} \langle \phi_q(\overrightarrow{\sigma}^{(\V)}),\overrightarrow{\sigma}^{(\V)}\rangle_{\widetilde{S}_q^{(\V)}} + \sum_{q=0}^{d-1} (-1)^q  \sum_{\alpha\in U_q^{(\V)}} \langle \phi^*_q(\widetilde{\alpha}),\widetilde{\alpha}\rangle_{\widetilde{S}_q^{(\V)}}\\ &+  \sum_{q=1}^d (-1)^q \sum_{\beta\in D_q^{(\V)}} \langle  \phi^*_q(\beta), \beta \rangle_{\widetilde{S}_q^{(\V)}}\\
			= & \sum_{q=0}^d (-1)^q  \sum_{\sigma\in\crit_q^{(\V)}(\K)} \langle \phi_q(\overrightarrow{\sigma}^{(\V)}),\overrightarrow{\sigma}^{(\V)}\rangle_{\widetilde{S}_q^{(\V)}} + \sum_{q=1}^{d} (-1)^{q-1}  \sum_{\alpha\in U_{q-1}^{(\V)}} \langle \phi^*_{q-1}(\widetilde{\alpha}),\widetilde{\alpha}\rangle_{\widetilde{S}_{q-1}^{(\V)}} \\
			&+  \sum_{q=1}^d (-1)^q \sum_{\beta\in D_q^{(\V)}} \langle  \phi^*_q(\beta), \beta \rangle_{\widetilde{S}_q^{(\V)}}\\
			= & \sum_{q=0}^d (-1)^q  \sum_{\sigma\in\crit_q^{(\V)}(\K)} \langle \phi_q(\overrightarrow{\sigma}^{(\V)}),\overrightarrow{\sigma}^{(\V)}\rangle_{\widetilde{S}_q^{(\V)}}\\
			& + \sum_{q=1}^{d} (-1)^{q-1} \left( \sum_{\alpha\in U_{q-1}^{(\V)}} \langle \phi^*_{q-1}(\widetilde{\alpha}),\widetilde{\alpha}\rangle_{\widetilde{S}_{q-1}^{(\V)}} - \sum_{\beta\in D_q^{(\V)}} \langle  \phi^*_q(\beta), \beta \rangle_{\widetilde{S}_q^{(\V)}} \right)\\
			= & \sum_{q=0}^d (-1)^q  \sum_{\sigma\in\crit_q^{(\V)}(\K)} \langle \phi_q(\overrightarrow{\sigma}^{(\V)}),\overrightarrow{\sigma}^{(\V)}\rangle_{\widetilde{S}_q^{(\V)}} \\
			&+ \sum_{q=1}^{d} (-1)^{q-1} \left( \sum_{\beta\in D_{q}^{(\V)}} \langle \phi^*_{q-1}(\partial(\beta)),\partial(\beta)\rangle_{\widetilde{S}_{q-1}^{(\V)}} - \sum_{\beta\in D_q^{(\V)}} \langle  \phi^*_q(\beta), \beta \rangle_{\widetilde{S}_q^{(\V)}} \right)\\
			= & \sum_{q=0}^d (-1)^q  \sum_{\sigma\in\crit_q^{(\V)}(\K)} \langle \phi_q(\overrightarrow{\sigma}^{(\V)}),\overrightarrow{\sigma}^{(\V)}\rangle_{\widetilde{S}_q^{(\V)}} \\
			&+ \sum_{q=1}^{d} (-1)^{q-1} \sum_{\beta\in D_{q}^{(\V)}} \left( \langle \phi^*_{q-1}(\partial(\beta)),\partial(\beta)\rangle_{\widetilde{S}_{q-1}^{(\V)}} - \langle  \phi^*_q(\beta), \beta \rangle_{\widetilde{S}_q^{(\V)}} \right)\\
			= & \sum_{q=0}^d (-1)^q  \sum_{\sigma\in\crit_q^{(\V)}(\K)} \langle \phi_q(\overrightarrow{\sigma}^{(\V)}),\overrightarrow{\sigma}^{(\V)}\rangle_{\widetilde{S}_q^{(\V)}} + 0 ~ \text{ (from Lemma~\ref{3.4.})}\\
			= & \sum_{q=0}^d (-1)^q  \sum_{\sigma\in\crit_q^{(\V)}(\K)} \langle \phi_q(\overrightarrow{\sigma}^{(\V)}),\overleftarrow{\sigma}^{(\V)}\rangle_{S_q} ~~\text{(from Lemma~\ref{l-3.4})}.
			\end{align*}

	\end{proof}
    \section{Combinatorial degree versions of generalized $\mathbb{Z}_p$- Tucker's lemma} \label{s4}
         We note that, an $d$-dimensional pseudomanifold $\s$ is a combinatorial sphere, if and only if there exists a maximal simplex $\sigma\in \s$, such that, $\s \setminus \{\sigma\}$ is collapsible.
         Now let us prove the following lemmas regarding the simplicial complex $\Delta^n$. 
     \begin{lemma} \label{l-4.2} For any maximal simplex $\sigma \in \Delta^n_{n-1}$ (the $(n-1)$-skeleton of $\Delta^n$), $\Delta^n_{n-1} \setminus \{\sigma\}$ is collapsible.
    \end{lemma}
    \begin{proof}
    	Let $V(\Delta^n_{n-1})=\{a_0,a_1,\ldots, a_n\}$. Without loss of generality, let $\sigma=\{a_0,a_1,\ldots,a_{n-1}\}$. We define a discrete vector field $\V$ on $V(\Delta^n_{n-1})$ as follows.
    	\[\V= \{(\alpha, \alpha \cup \{a_n\}): \alpha \in \Delta^n_{n-1} \text{ and } a_n\notin \alpha\}.\]
    	One can check that $\V$ is a gradient vector field on $\Delta^n_{n-1} \setminus \{\sigma\}$, where the only critical simplex is the $0$-simplex $\{a_n\}$.Hence, $\Delta^n_{n-1} \setminus \{\sigma\}$ is collapsible.
    \end{proof}
   
      \begin{lemma}\label{l4.3} For any maximal simplex $\sigma \in \Bd(\Delta^n_{n-1})$, $\Bd(\Delta^n_{n-1}) \setminus \{\sigma\}$ is collapsible. 
    \end{lemma}
    \begin{proof}
    	We prove this lemma by induction on the dimension $n$. When $n=1$, the result is clearly true. 
    	
    	Let us assume that our hypothesis is true for dimension $(n-1)$.
    	
       Now, let us consider $\Bd(\Delta^n_{n-1})$. Let $\sigma$ be a maximal simplex of $\Bd(\Delta^n_{n-1})$. Suppose that, $\sigma = (v_{\tau_{0}} v_{\tau_{1}}\ldots v_{\tau_{n-1}})$,
    	where $\tau_0,\tau_1,\ldots,\tau_{n-1} \in \Delta^n_{n-1}$ with $\tau_0^{(0)}\subsetneq \tau_1^{(1)} \subsetneq \ldots \subsetneq \tau_{n-1}^{(n-1)}$. Let, 
    	$\sigma^\prime= (v_{\tau_{0}} \ldots v_{\tau_{n-2}})$.
    	 
    	Now, observe that
    	\[\operatorname{Bd}(\Delta^n_{n-1})\setminus \{\sigma\}= \operatorname{Bd}(\Delta^n_{n-1}\setminus \{\tau_{n-1}\}) \cup (v_{\tau_{n-1}}*(\Bd(\operatorname{cl}(\tau_{n-1})\setminus\{\tau_{n-1}\})\setminus \{\sigma^\prime\})).\]
    	
    	Let $\K_1= \operatorname{Bd}(\Delta^n_{n-1}\setminus \{\tau_{n-1}\})$, $\K_2=\Bd(\operatorname{cl}(\tau_{n-1})\setminus\{\tau_{n-1}\})\setminus \{\sigma^\prime\}$, and $x= v_{\tau_{n-1}}$. Clearly, $\K_2$ is a subcomplex of $\K_1$. Now, by induction $\K_2$ is collapsible. Also note that, $x \notin V(\K_1)$.
    	Hence, by Corollary~\ref{c-4.4}, $\operatorname{Bd}(\Delta^n_{n-1})\setminus \{\sigma\} (= \K_1 \cup (x*\K_2))$ collapses into $\K_1 (=\operatorname{Bd}(\Delta^n_{n-1}\setminus \{\tau_{n-1}\}) )$. Now, by Lemma~\ref{l-4.2}, $\Delta^n_{n-1}\setminus \{\tau_{n-1}\}$ is collapsible, and by Theorem~\ref{T2.21},  $\operatorname{Bd}(\Delta^n_{n-1}\setminus \{\tau_{n-1}\})$ is also collapsible. Hence, $\operatorname{Bd}(\Delta^n_{n-1})\setminus \{\sigma\}$ is collapsible.
    \end{proof}
      
      We show that the barycentric subdivision of a combinatorial sphere is a combinatorial sphere.
    
    \begin{proposition} \label{P4.4}
    	For a combinatorial sphere $\s$, $\operatorname{Bd}(\s)$ is also a combinatorial sphere.
    	\end{proposition}
    \begin{proof}
    	Let $\dim(\s)=d$. Since, $\s$ is a combinatorial sphere, there exists a $d$-simplex $\sigma \in \s$, such that, $\s \setminus \{\sigma\}$ is collapsible.  Let $\tau \in \s_d(\Bd(\s))$, such that $v_\sigma \in \tau$, where $v_\sigma$ is the vertex (of $\Bd(\s))$ corresponding to the $d$-simplex $\sigma$. Now we show that $\Bd(\s)  \setminus \tau$ is collapsible.
    	
    	Let $\eta= \tau \setminus \{v_{\sigma}\}$. Now, observe that
    	
    	\[\Bd(\s)\setminus \{\tau\}=\operatorname{Bd}(\s \setminus \{\sigma\}) \cup (v_{\sigma} *(\Bd(\cl(\sigma)\setminus \{\sigma \})\setminus \{\eta\})).\]
    	
    	Let $\K_1= \operatorname{Bd}(\s \setminus \{\sigma\}) $, $\K_2= \Bd(\cl(\sigma)\setminus \{\sigma \})\setminus \{\eta\}$, and $x= v_\sigma$. Clearly, $\K_2$ is a subcomplex of $\K_1$. Now, by Lemma~\ref{l4.3}, $\K_2$ is collapsible.  Also, note that $x \notin V(\K_1)$.
    	Hence, by Corollary~\ref{c-4.4}, $\operatorname{Bd}(\s)\setminus \{\tau\} (= \K_1 \cup (x*\K_2))$ collapses into $\K_1 (=\operatorname{Bd}(\s\setminus \{\sigma \}) )$. Now, by Theorem~\ref{T2.21},  $\operatorname{Bd}(\s\setminus \{\sigma\})$ is collapsible. Hence, $\operatorname{Bd}(\s)\setminus \{\tau\}$ is collapsible.
    	
    	Now, one may check that the barycentric subdivision of a pseudomanifold is also a pseudomanifold. Hence, $\Bd(\s)$ is a combinatorial sphere
    	
    \end{proof}
   
  Now we introduce the notion the notion of the \emph{degree} of a chain map.
  
  \begin{definition}[Degree of a chain map]
	Let $\s$ be a combinatorial $d$-sphere, and $\phi_\#:C_\#(\s)  \rightarrow C_\#(\s)$ be a chain map. Let $r$  be an orientation on $\s$, and $[\s_r]$ be the fundamental cycle corresponding to $r$. Then 
	\[\phi_d([\s_r])=m\cdot [\s_r],\]
	for some $m \in \mathbb{Z}$. This integer $m$
	is said to be the degree of the chain map $\phi_\#$.
  \end{definition}
    The well-definedness of this notion of degree is proved in Proposition~\ref{A5}.
    
    \begin{definition}[Combinatorial degree] \label{d4.6}
    	
    	Let $\s$ be combinatorial $d$-sphere and $k$ be a non-negative integer. Let $\Sd^k_\#: C_\# (\Bd^k(\s)) \rightarrow C_\#(\Bd^k(\s))$ be the $k$-th subdivision chain map.  Then,  for a simplicial map $f: \Bd^k(\s)\rightarrow \s$, we define the \emph{combinatorial degree} of the simplicial map $f$, $\deg(f)$, to be the degree of the chain map $\Sd^k_\# \circ f_\#:  C_\#(\Bd^k(\s))\rightarrow C_\#(\Bd^k(\s))$, where $f_\#: C_\#(\Bd^k(\s)) \rightarrow C_\#(\s)$ is the chain map induced by the simplicial map $f$.

    \end{definition}
 
   \begin{lemma} \label{l4.7}
   	Let $\s$ be a combinatorial $d$-sphere, and $\V$ be a gradient vector field on $\s$, having exactly two critical simplices, one of dimension $d$ and one of $0$. Let $\tau$ be the critical simplex of dimension $d$. Then $\partial(\overrightarrow{\tau}^{(\V)})=0$.
   \end{lemma}
   \begin{proof}
   	By Lemma~\ref{l3.3}, $\partial(\overrightarrow{\tau}^{(\V)})$ is generated by the $(d-1)$-dimensional critical chains.
   	
   	\textbf{Case 1:} When $d> 1$, there is no $(d-1)$-dimensional critical chain. Hence $\partial(\overrightarrow{\tau}^{(\V)})=0$.
   	
   	\textbf{Case 2:} ($d=1$), Let $\s$ be a combinatorial $1$-sphere, with a gradient vector field $\V$, having exactly two critical simplices, $\tau(=(v_iv_{i+1})$ and $(v_0)$ (see Figure~\ref{F3}). Then $\partial(\overrightarrow{\tau}^{(\V)})$ is generated by the $0$-dimensional critical chain $\overrightarrow{(v_0)}^{(\V)}$, which is the same as $(v_0)$.

   	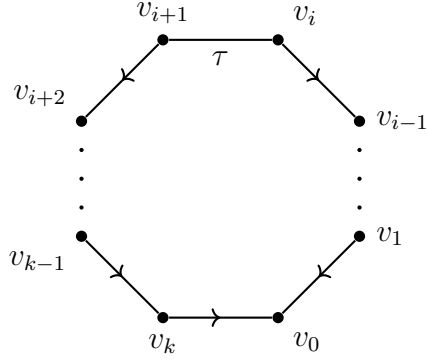
\begin{figure}[h]
   	\centering
   	\begin{tikzpicture}[scale=2,
   		% define style for arrows in the middle
   		middle arrow/.style={
   			postaction={decorate},
   			decoration={
   				markings,
   				mark=at position 0.5 with {\arrow{>}}
   			}
   		}
   		]

   		% Define octagon vertices explicitly (unit circle coordinates)
   		\node[circle, fill=black, inner sep=1.5pt, label=right:$v_{i-1}$] (vi-1) at (0.9239,0.3827) {};
   		\node[circle, fill=black, inner sep=1.5pt, label=above right:$v_i$] (vi) at (0.3827,0.9239) {};
   		\node[circle, fill=black, inner sep=1.5pt, label=above:$v_{i+1}$] (vi+1) at (-0.3827,0.9239) {};
   		\node[circle, fill=black, inner sep=1.5pt, label=above left:$v_{i+2}$] (vi+2) at (-0.9239,0.3827) {};
   		\node[circle, fill=black, inner sep=1.5pt, label=below left:$v_{k-1}$] (vk-1) at (-0.9239,-0.3827) {};
   		\node[circle, fill=black, inner sep=1.5pt, label=below:$v_k$] (vk) at (-0.3827,-0.9239) {};
   		\node[circle, fill=black, inner sep=1.5pt, label=below right:$v_0$] (v0) at (0.3827,-0.9239) {};
   		\node[circle, fill=black, inner sep=1.5pt, label=right:$v_1$] (v1) at (0.9239,-0.3827) {};
   		\node at (0,0.8) {$\tau$};
   		
   		\draw[thick, middle arrow] (vk-1) -- (vk);
   		\draw[thick, middle arrow] (v1) -- (v0);
    	\draw[thick, middle arrow] (vk) -- (v0);  
    	\draw[thick, middle arrow] (vi+1) -- (vi+2);		
    	\draw[thick, middle arrow] (vi) -- (vi-1);
    	\draw[thick] (vi) -- (vi+1);
    	
    	 \foreach \t in {0.25,0.5,0.75} {
    		\fill ($(vi+2)!\t!(vk-1)$) circle (.4pt);
    		\fill ($(vi-1)!\t!(v1)$) circle (.4pt);
    	}
   	\end{tikzpicture}
   	\caption{The combinatorial $1$-sphere $\s$ and gradient vector field $\V$  with $\tau$ and $v_0$ are critical (for example, `$v_i \rightarrow v_{i-1}$' implies $((v_i)^{(0)},(v_iv_{i-1})^{(1)})\in \V$). }	\label{F3}
   	\end{figure}
   		
   		Now, the only two 1-simplices that contain $(v_0)$ are $(v_0v_k)$ and $(v_0v_1)$. There is a unique $\V$-trajectory $P_1$ starting from $\tau$ and ending at $(v_0v_1)$, which is
   		\[P_1: (v_iv_{i+1}) \rightarrow (v_i) \rightarrowtail (v_{i-1}v_i) \rightarrow \cdots \rightarrow (v_2) \rightarrowtail (v_1v_2) \rightarrow (v_1) \rightarrowtail (v_0v_1).\]
   		There is also a  unique $\V$-trajectory starting from $\tau$ and ending at $(v_0v_k)$, which is
   		\[P_2: (v_iv_{i+1}) \rightarrow (v_{i+1}) \rightarrowtail (v_{i+1}v_{i+2}) \rightarrow \cdots \rightarrow (v_{k-1}) \rightarrowtail (v_{k-1}v_k) \rightarrow (v_k) \rightarrowtail (v_0v_k).\]
   		Moreover, we have
   		\begin{align*}
   			w(P_1)&= \prod_{t=1}^i(-[(v_tv_{t+1}),(v_t)]\cdot[(v_{t-1}v_t),(v_t)])\\ 
   			&=\prod_{t=1}^i(-(+1)\cdot(-1))\\
   			&=1,
   		\end{align*}
   		and
   			\begin{align*}
   			w(P_2)&= \left(\prod_{t=i+1}^{k-1}(-[(v_{t-1}v_{t}),(v_t)]\cdot[(v_{t}v_{t+1}),(v_t)])\right)\times (-[(v_{k-1}v_{k}),(v_k)]\cdot[(v_{0}v_k),(v_k)])\\ 
   			&=\left(\prod_{t=1}^i(-(-1)\cdot(1))\right) \times (-(1)\cdot(1))\\
   			&=(1) \times (-1)\\
   			&=-1.
   		\end{align*}
   	    	Now the coefficient of $v_0$ in $\partial(\overrightarrow{\tau}^{(\V)})$ is,
   	    \begin{align*}  \label{E14}
   	    	\langle \overrightarrow{\tau}^{(\V)}, v_0\rangle &=\sum_{\substack{P: P \text{ is a}\\ \V\text{-trajectory},\\ \inn(P)=\sigma}} w(P)\cdot [ \trr(P), v_0]\\
   	    	&= w(P_1)\cdot[(v_0v_1), (v_0)] + w(P_2)\cdot[(v_0v_k), (v_0)]\\
   	    	&= 1\times 1 + (-1)\times 1\\
   	    	&=0
   	    \end{align*}
   		Hence, $\partial(\overrightarrow{\tau}^{(\V)})=0$.
   		
   	\end{proof}
   	
   	\begin{remark} \label{R4.8}
   		  Let $\s$ be a combinatorial $d$-sphere and $k$ be a non-negative integer. Then, by Proposition~\ref{P4.4}, $\Bd^k(\s)$ is also a combinatorial $d$-sphere. Let $\V$ be a gradient vector field on $\Bd^k(\s)$ having exactly two critical simplices, one of dimension $d$ and one of dimension $0$. Let $\tau$ be the $d$-dimensional critical simplex. Then for any simplicial map $f: \Bd^k(\s)\rightarrow (\s)$ of degree $m$, we can write $(\Sd^k_d \circ f_d)(\overrightarrow{\tau}^{(\V)})=m\cdot \overrightarrow{\tau}^{(\V)}$ (By Lemma~\ref{l4.7} and the well-definedness of degree (see  Appendix~\ref{Appendix})). Indeed, $\overrightarrow{\tau}^{(\V)}$ can be realised as a fundamental cycle corresponding to an orientation of $\Bd^k({\s})$ (see the proof of Theorem~\ref{A2}).
   	\end{remark}

   \begin{lemma} \label{L4.9}
   	Let $\s$ be a combinatorial $d$-sphere and $\V$ be a gradient vector field on $\s$ having exactly two critical simplices, one of dimension $0$ and one of dimension $d$. Let $v$ be the critical $0$-simplex. Then, for every vertex $u\in S_0(\s)$, there exists a unique co-$\V$-trajectory $Q$ from $v$ to $u$. Moreover, $w(Q)=1$.
   \end{lemma}
   \begin{proof} If $u=v$, the the trivial trajectory $Q:v$ is the unique co-$\V$-trajectory from $v$ to $u$. Moreover, we have $w(Q)=1$.
   	If $u\neq v$, let us consider the sequence
   	\[(u=)~v_0, \V(v_0), v_1, \text{ where, } v_1~(\neq v_0) \in \V(v_0).\]
   	 Now, if $v_1=v$, the following is a co-$\V$-trajectory from $v$ to $u$.
   	\[(v=)~v_1 \leftarrow \V(v_0) \leftarrowtail v_0~(=u).\]
   	If $v_1\neq v$, then $\V(v_1)$ exists, and we consider the sequence
   	\[(u=)~v_0, \V(v_0), v_1,\V(v_1), v_2, \text{ where, } v_2~(\neq v_1) \in \V(v_1).\]
    Now, if $v_2=v$, the following is a co-$\V$-trajectory from $v$ to $u$.
   	\[(v=)~v_2 \leftarrow \V(v_1) \leftarrowtail v_1^{(0)} \leftarrow \V(v_0) \leftarrowtail v_0~(=u).\]
   		If $v_2\neq v$, $\V(v_2)$ exists, and we proceed similarly.
   		Now, since $S_0(\s)$ is finite and $\V$ is acyclic, this process terminates, and for some $t~(\geq3)$, we get $v_t=v$. In that case the following is the desired co-$\V$-trajectory.
   			\[Q:(v=)~v_t \leftarrow \V(v_{t-1}) \leftarrowtail v_{t-1} \leftarrow \cdots \leftarrowtail v_2 \leftarrow \V(v_1) \leftarrowtail v_1 \leftarrow \V(v_0) \leftarrowtail v_0~(=u).\]
   We note that, in the co-$\V$-trajectory $Q$, for each $\V(v_{i-1})$, $i\in \{1,\ldots,t\}$, the choice of $v_i$ is unique. So, $Q$ is the unique co-$\V$-trajectory from $v$ to $u$. Moreover, we have
   \begin{align*}
   	w(Q)&= \prod_{i=0}^{t-1}( -[\V(i), v_{i+1}]\cdot [\V(i), v_{i}])\\
   	&=  \prod_{i=0}^{t-1}( -(-1))\\
   	&=1.
   \end{align*}
\end{proof}
   The following is a immediate corollary of Lemma~\ref{L4.9}.

   \begin{corollary}\label{C4.8}
   	Let $\s$ be a combinatorial $d$-sphere and $\V$ be a gradient vector field on $\s$ having two critical simplices, one of dimension $0$ and one of dimension $d$. Let $v$ be the critical $0$-simplex. Then, for each $u\in S_0(\s)$, $\langle\overleftarrow{v}^{(\V)},u	\rangle =1$.
   	\end{corollary}
   	\begin{proof}
   		We have,
   		\begin{align*}
   		\overleftarrow{v}^{(\V)} & = \sum_{\substack{Q: Q \text{ is a}\\ \text{co-}\V\text{-trajectory},\\ \inn(Q)=v}} w(Q)\cdot \trr(Q) ~~~ \text{ (Definition~\ref{D1.6})}\\
   		 &=\sum_{u\in S_0(\s)} \left(\sum_{\substack{Q: Q \text{ is a}\\ \text{co-}\V\text{-trajectory},\\ \text{from } v \text{ to }u.}} w(Q)\right) u\\
   		&=\sum_{u\in S_0(\s)} u~~~\text{ (by Lemma~\ref{L4.9})}.
   	    \end{align*}
   	Which implies, $\langle\overleftarrow{v}^{(\V)},u	\rangle =1$, for each $u\in S_0(\s)$.
   		
   	\end{proof}

      The following lemma is a well-known topological result, whose standard proof uses an  Euler characteristic based argument. However, in our combinatorial setup, we provide a proof using the combinatorial Hopf trace formula (which is bit of an overkill though). 
    \begin{lemma} \label{l4.9}
    	Let $\s$ be a $d$-dimensional $\mathbb{Z}_p$-combinatorial sphere. If $p > 2$,  $d$ is odd.
    \end{lemma}
    
    \begin{proof}
    	Let $\V$ be a gradient vector field on $\s$, such that, there exists only two critical simplices, one of dimension $d$, say $\sigma_0$, and one of dimension $0$, say $v$. Let $\Id_\#: C_\#(\s, \mathbb{Z})\rightarrow C_\#(\s,\mathbb{Z})$ be the identity chain map. Then from Theorem~\ref{t1.10}, we have 
    	\begin{align} \label{E18}
    	\nonumber	\sum_{q=0}^d (-1)^q \operatorname{tr}(\Id_q)&=\sum_{q=0}^d (-1)^q \sum_{\sigma\in \crit_q^{(\V)}(\K)} \langle \overleftarrow{\sigma}^{(\V)},\Id_q(\overrightarrow{\sigma}^{(\V)})
    	\rangle_{S_q}\\
    	\nonumber&=\sum_{q=0}^d (-1)^q \sum_{\sigma\in \crit_q^{(\V)}(\K)} \langle \overleftarrow{\sigma}^{(\V)},\overrightarrow{\sigma}^{(\V)}
        \rangle_{S_q}\\
    	\nonumber&= (-1)^d \cdot \langle \overleftarrow{\sigma_0}^{(\V)},\overrightarrow{\sigma_0}^{(\V)}
    	\rangle_{S_q} + (-1)^0 \langle \overleftarrow{v}^{(\V)},\overrightarrow{v}^{(\V)}
    	\rangle_{S_q}\\
    	\nonumber&= (-1)^d \cdot \langle \sigma_0,\overrightarrow{\sigma_0}^{(\V)}
    	\rangle_{S_q} + (-1)^0 \langle \overleftarrow{v}^{(\V)},v
    	\rangle_{S_q}\\
    	&= (-1)^d + 1
       \end{align} 
       Now since, $\s$ is a $\mathbb{Z}_p$-combinatorial sphere, $p$ divides the left-hand side of the Equation~\eqref{E18}, and hence the right-hand side. Which implies $d$ is odd.
 
       \end{proof}
	  
    Now we are in a position to prove Theorem~\ref{t1.7}.
    \begin{proof} [Proof of Theorem~\ref{t1.7}]
    	Since, $\s$ is a combinatorial $d$-sphere, by Proposition~\ref{P4.4}, $\Bd^k(\s)$ is also a combinatorial $d$-sphere. Let $\V$ be a gradient vector field on $\Bd^k(\s)$ having exactly two critical simplices, one $d$-dimensional, say $\tau$, and one $0$-dimensional, say $v$. Let $f_\#:C_\#(\Bd^{k}(\s), \mathbb{Z}) \rightarrow C_\#(\s, \mathbb{Z})$ be the chain map induced by the simplicial map $f: \operatorname{Bd}^{k}(\s) \rightarrow \s $, and $\Sd^k_\#: C_\#(\s, \mathbb{Z}) \rightarrow C_\#(\Bd^k(\s),\mathbb{Z})$ be the $k$-th subdivision map. Let $\phi_\#=\Sd^k_\# \circ f_\#$. Then,  for the chain map $\phi_\#: C_\#(\Bd^k(\s), \mathbb{Z}) \rightarrow C_\#(\Bd^k(\s),\mathbb{Z})$, by Theorem~\ref{t1.10}, we have
	\begin{equation} \label{Eq14}
		\sum_{q=0}^d (-1)^q \operatorname{tr}(\phi_q)=(-1)^d \langle\overleftarrow{\tau}^{(\V)},\phi_d(\overrightarrow{\tau}^{(\V)})
		\rangle + (-1)^0 \langle\overleftarrow{v}^{(\V)},\phi_0(\overrightarrow{v}^{(\V)})
		\rangle.
	\end{equation} 

Since, $\Bd^k(\s)$ is a $\mathbb{Z}_p$-combinatorial sphere, $p$ divides the left-hand side of the Equation~\eqref{Eq14}, and hence the right-hand side too. So, we have
\begin{align*}
	& (-1)^d \langle\overleftarrow{\tau}^{(\V)},\phi_d(\overrightarrow{\tau}^{(\V)})
	\rangle + (-1)^0 \langle\overleftarrow{v}^{(\V)},\phi_0(\overrightarrow{v}^{(\V)})
	\rangle \equiv 0 \pmod{p}\\
	\implies & (-1)^d \langle \overleftarrow{\tau}^{(\V)} ,\deg(f)\cdot \overrightarrow{\tau}^{(\V)}
	\rangle + \langle\overleftarrow{v}^{(\V)},\phi_0(\overrightarrow{v}^{(\V)})
	\rangle \equiv 0 \pmod{p} ~~\text{ (see Remark~\ref{R4.8})}\\
	\implies & (-1)^d \langle\tau,\deg(f)\cdot \overrightarrow{\tau}^{(\V)}
	\rangle + \langle\overleftarrow{v}^{(\V)},\phi_0(v)
	\rangle \equiv 0 \pmod{p} \\
	\implies & (-1)^d \cdot \deg(f)\cdot \langle\tau, \overrightarrow{\tau}^{(\V)}
	\rangle + \langle\overleftarrow{v}^{(\V)},\phi_0(v)
	\rangle \equiv 0 \pmod{p} \\
	\implies & (-1)^d \cdot \deg(f) + 1 \equiv 0 \pmod{p} ~\text{ (by Corollary~\ref{C4.8})}
\end{align*}
    Now, when $p=2$, we have 
    \begin{align*}
    	& (-1)^d \cdot \deg(f) + 1 \equiv 0 \pmod{2} \\
       \implies & \deg(f) \equiv 1 \pmod{2}. 
    \end{align*} 
When $p>2$, by Lemma~\ref{l4.9}, $d$ is odd. Hence, we have 
\begin{align*}
 & (-1)^d \cdot \deg(f) + 1 \equiv 0 \pmod{p} \\
\implies & -\deg(f) + 1 \equiv 0 \pmod{p} \\
\implies & \deg(f) \equiv 1 \pmod{p} \\
\end{align*}
Hence, the theorem is proved.
    \end{proof}
    
    Now we need the following proposition to prove the Theorem~\ref{t1.9}. 

    \begin{proposition} \label{p4.12}
    	Join of two combinatorial spheres is also a combinatorial sphere.
    \end{proposition}
    \begin{proof}
    	Let $\mathcal{C}$ and $\mathcal{D}$ be two combinatorial spheres of dimension $m$ and $n$, respectively. Let $\V_{\mathcal{C}}$ be a   gradient vector field on $\mathcal{C}$, such that, the only critical simplices of $\mathcal{C}$ corresponding to $\V_{\mathcal{C}}$ are $\gamma_1$ and $\gamma_2$, with $\dim(\gamma_1)=0$ and $\dim(\gamma_2)=m$. Similarly,  $\V_{\mathcal{D}}$ be a gradient vector field on $\D$, such that, the only critical simplices of $\mathcal{D}$ corresponding to $\V_{\mathcal{D}}$ are $\delta_1$ and $\delta_2$, with $\dim(\delta_1)=0$ and $\dim(\delta_2)=n$. We have
    	\[\C * \D= \{\sigma \sqcup \tau : \sigma \in \C \text{ and } \tau \in \D\}.\]
    	One can easily verify that join of two pseudomanifold is also a pseudomanifold. So, $\C * \D$ is a $(m+n+1)$-dimensional pseudomanifold. Now, to show $\C * \D$ is a combinatorial sphere, we need to prove existence of a gradient vector field $\W$ on $\C * \D$, such that $\W$ has only two critical simplices, one $(m+n+1)$-dimensional and one $0$-dimensional. We construct a discrete vector field $\W$ on $\C*\D$ in the following steps.
    	
    	\textbf{Step 1:} First, we pair the simplices of the form $\sigma \sqcup \emptyset$ with $\sigma \sqcup \delta_1$. That means, in this step we have the following set of pairs.
    	\[\W_1=\{(\sigma \sqcup \emptyset, \sigma \sqcup \delta_1): \sigma \in \C \}.\]
    	Note that, in this process the empty simplex $\emptyset $ (i.e., $\emptyset \sqcup \emptyset$) in $\C* \D$ is paired with the simplex $\delta_1$ (i.e., $\emptyset \sqcup \delta_1$). 
    	
    	\textbf{Step 2:} Next, we construct the following set of pairings. 
    	\[\W_2=\{(\emptyset \sqcup \tau, \emptyset \sqcup \V_{\D}(\tau)): \tau(\neq \emptyset) \in U^{(\V_D)}(\D)\}.\]
    	 Note that, for $\tau=\emptyset$, the simplex $\emptyset \sqcup \emptyset$ is already paired with $\emptyset \sqcup \delta_1$ in previous step.
    	 The only simplex in $\C * \D$ of the form $\emptyset \sqcup \tau$, which is still unpaired, is $\emptyset \sqcup \delta_2$.
    	 
    	\textbf{Step 3:} Next, we construct the following set of pairings.
     	\[\W_3= \{ (\sigma \sqcup \tau, \V_{\C}(\sigma) \sqcup \tau): \sigma(\neq \emptyset) \in U^{(\V_{\C})} (\C), \text{ and }  \tau (\neq \emptyset,\delta_1) \in \D\}.\]
     	Observe that, if $\tau=\delta_1$ then the simplex $\sigma \sqcup \delta_1$ is paired with the simplex $\sigma \sqcup \emptyset$ in Step~1.
     	
     	\textbf{Step 4:} Next, we construct the following set of pairings.
     	\[\W_4=\{(\sigma \sqcup \tau, \sigma \sqcup \V_{D}(\tau)): \sigma \in \crit^{(\V_{\C})}(\C), \text{ and } \tau  (\neq \emptyset) \in U^{(\V_{\D})}(\D)\}.\]
        We note that, After Step~4, the unpaired simplices are $\emptyset \sqcup \delta_2, \gamma_1 \sqcup \delta_2, \gamma_2 \sqcup \delta_2$. 	
     	
     	\textbf{Step 5:} Finally, we make the following pair.
     	\[\W_5= \{(\emptyset \sqcup \delta_2, \gamma_1 \sqcup \delta_2)\}.\]
     	
     	 We set, $\W=\W_1\sqcup \W_2 \sqcup \ldots \sqcup \W_5$. Note that, the only unpaired simplex in $\W$ is $\gamma_2 \sqcup \delta_2$. We can easily check that $\W$ is well-defined as a discrete vector field. Now, we show that $\W$ is a gradient vector field. 
     	  
     	\textbf{Acyclicity of $\W$:} We will prove the acyclicity of $\W$ by contradiction. Let us assume the following is a closed $\W$-trajectory in $\C * \D$.
     	\[P: \beta_0^{(q)} \rightarrow \alpha_1^{(q-1)} \rightarrowtail \beta_1^{(q)} \rightarrow \alpha_2^{(q-1)} \rightarrowtail \beta_2^{(q)} \rightarrow \cdots \rightarrow \alpha_r^{(q-1)} \rightarrowtail \beta_r^{(q)}(=\beta_0)\]
     	
     	\textbf{Case 1:} Let $\alpha_1= \sigma_1 \sqcup \emptyset$, for some $\sigma_1 \in \C$. Then $\beta_1=\sigma_1 \sqcup \delta_1$ (by Step~1). So, $\alpha_2= \sigma \sqcup \delta_1$, for some facet $\sigma$ of $\sigma_1$, and hence, $\alpha_2$ is paired with a lower dimensional simplex $\sigma \sqcup \emptyset$ (by Step~1). This is a contradiction.
     	
     	\textbf{Case 2:} Let $\alpha_1= \emptyset \sqcup \tau_1$, for some $\tau_1 \in \D$. Note that, $\tau_1$ can not be $\delta_1$ or paired with a lower dimensional simplex in $\V_{\D}$, otherwise $\alpha_1$ is paired with a lower dimensional simplex in $\W$ (by Step~1 and Step~2). Now consider the following subcases.
     	\begin{enumerate}[(i)]
     		
     	\item If $\tau_1=\delta_2$, i.e., $\alpha=\emptyset \sqcup \delta_2$, then $\beta_1=\gamma_1 \sqcup \delta_2$ (by Step~5), and $\alpha_2=\gamma_1 \sqcup \tau_2$, for some facet $\tau_2$ of $\delta_2$. Now, if $\tau_2 = \delta_1$, or $\tau_2$ is paired 
     	with a lower dimensional simplex in $\V_{\D}$, $\alpha_2$ is paired with a lower dimensional simplex (by Step~1 and Step~4), a contradiction. Hence, $\tau_2 \in U^{(\V_\D)}(\D)$, and $\beta_2= \gamma_1 \sqcup \V_{\D}(\tau_2)$ (by Step~4). Now since, $\alpha_1=\emptyset \sqcup \delta_2$, and the trajectory $P$ is closed, there exists a smallest $k \geq 2$, such that, $\beta_k=\gamma_1 \sqcup \V_{\D}(\tau_k)$, for some $\tau_k \in \D$, and $\alpha_{k+1}=\emptyset \sqcup \V(\tau_k)$. In this case, $\alpha_{k+1}$ is paired with a lower dimensional simplex $\emptyset \sqcup \tau_k$, which is a contradiction.

     	\item  If $\tau_1\neq \delta_2$, $\tau_1 \in U^{(\V_{\D})}(\D)$ and $\beta_1=\emptyset \sqcup \V_{\D}(\tau_1)$ (by Step~2). This implies, $\alpha_2= \emptyset \sqcup \tau_2$, for some facet $\tau_2$ of $\V_{\D}(\tau)$. Now, $\tau_2$ can not be paired	with a lower dimensional simplex in $\V_{\D}$, otherwise $\alpha_2$ is paired with a lower dimensional simplex (by Step~2), a contradiction.  Hence, $\tau_2 \in U^{(\V_\D)}(\D)$, and $\beta_2= \gamma_1 \sqcup \V_d(\tau_2)$ (by Step~4). Note that, the same pattern continues, and hence each simplex of the trajectory $P$ is of the form $\emptyset \sqcup \tau$, where $\tau \in \D$. This implies that, the simplices in the trajectory $p$ are in a natural one-to-one correspondence ($\emptyset \sqcup \tau \leftrightarrow \tau$) with the simplices in a $\V_{(\D)}$-trajectory, which can not be a closed trajectory (from the acyclicity of $\V_{\D}$), a contradiction.
     	
        \end{enumerate}

     	\textbf{Case 3:} Let $\alpha_1=\sigma_1 \sqcup \tau_1$, where $\sigma_1, \tau_1 \neq \emptyset$. Note that, $\tau_1  \neq \delta_1$, otherwise $\alpha$ is paired with a lower dimensional simplex $\alpha_1 \sqcup \emptyset$ (by Step~1), a contradiction. Now let us consider the following subcases.
     	
     	\begin{enumerate}[(i)]
     		\item \label{Sc1}	Let $\sigma_1 \neq \gamma_1$ or $\gamma_2$.  Then $\sigma\in U^{(\V_c)}(C)$, otherwise, if $\sigma_1\in D^{(\V_c)}(C)$), $\alpha_1$ is paired with a lower dimensional simplex (by Step~3), a contradiction. Therefore, we get $\beta_1=\V(\sigma_1) \sqcup \tau_1$ (by Step~3). Now, we claim that each simplex of the trajectory $P$ is of the form $\sigma \sqcup \tau_1$, for some $\sigma \in \C$. For otherwise, let $k~(\geq 1)$ be the smallest integer, such that, for each $1 \leq i \leq k$, $\alpha_i= \sigma_i \sqcup \tau_1$ and $\beta_i= \sigma_i^\prime \sqcup \tau_1$ ($\sigma_i$ is a facet of $\sigma_i^\prime$), and $\alpha_{k+1}= \sigma_k^\prime \sqcup \tau$, where $\tau$ is a facet of $\tau_1$. Now, observe that, for each $1 \leq i \leq k$, $\sigma_i^\prime=\V_{\C}(\sigma_i)$ (by Step~3). So, $\alpha_{k+1}= \V_{\C}(\alpha_k) \sqcup \tau$. Now, if $\tau=\delta_1$, then $\alpha_{k+1}$ is paired with a lower dimensional simplex $\V_{\C}(\alpha_k) \sqcup \emptyset$ (by Step~1), a contradiction. (If $\tau= \emptyset$, then $\beta_{k+1}= \V_{\C}(\alpha_k) \sqcup \delta_1$ (by Step~1), and $\alpha_{k+2}= \sigma^\prime \sqcup \delta_1$, where $\sigma^\prime$ is a facet of $\V_{\C}(\alpha_k)$. So, $\alpha_{k+2}$ is paired with a lower dimensional simplex $\sigma^\prime \sqcup \emptyset$ (by Step~1), a contradiction.) Now, $\alpha_{k+1}= \V_{\C}(\alpha_k) \sqcup \tau$, with $\tau \neq \emptyset$ or $\delta_1$, implies that  $\alpha_{k+1}$ is paired with a lower dimensional simplex $\sigma_k \sqcup \tau$ (by Step~3), a contradiction.  Hence, each simplex of the trajectory $P$ is of the form $\sigma \sqcup \tau_1$, for some $\sigma \in \C$. Therefore, each simplex of the trajectory $P$ is in a one-to-one correspondence ($\sigma \sqcup \tau_1 \leftrightarrow \sigma$) with the simplices of a $\V_{\C}$-trajectory, which can not be a closed trajectory (from the acyclicity of $\V_{\C}$), a contradiction.
     		
     		\item  \label{Sc2} Let $\sigma_1= \gamma_1$. Note that, $\tau_1 \neq \delta_2$, otherwise $\alpha_1$ is paired with a lower dimensional simplex $\emptyset \sqcup \delta_2$ (by Step~5), a contradiction. Also if $\tau_1 \in D^{(\V_{\D})}{\D}$, $\tau_1=\V_{\D}(\tau_0)$, for some $\tau_0 \in \D$, and  $\alpha_1$ is paired with a lower dimensional simplex $\gamma_1 \sqcup \tau_0$ (by Step~4), a contradiction. Hence, $\tau_1 \in U^{\V_{\D}}(\D)$, and therefore, $\beta_1=\gamma_1 \sqcup \V_{\D}(\tau_1)$ (by Step~4). Now we claim that each simplex of the trajectory $P$ is of the form $\gamma_1 \sqcup \tau$, for some $\tau \in \D$. For otherwise, let $k~(\geq 1)$ be the smallest integer, such that, for each $1 \leq i \leq k$, $\alpha_i= \gamma_1 \sqcup \tau_i$ and $\beta_i= \gamma_1 \sqcup \tau_i^\prime$ ($\tau_i$ is a facet of $\tau_i^\prime$), and $\alpha_{k+1}= \emptyset \sqcup \tau_k^\prime$. Now, observe that, for each $1 \leq i \leq k$, $\tau_i^\prime=\V_{\D}(\tau_i)$ (by Step~2). So, $\alpha_{k+1}= \emptyset \sqcup \V_{\D}(\tau_k)$, and it is paired with a lower dimensional simplex $\emptyset \sqcup \tau_k$, a contradiction. Hence, each simplex of the trajectory $P$ is of the form $\gamma_1 \sqcup \tau$, for some $\tau \in \D$. Therefore, each simplex of the trajectory $P$ is in a one-to-one correspondence ($\gamma_1 \sqcup \tau \leftrightarrow \tau$) with the simplices of a $\V_{\D}$-trajectory, which can not be closed (from the acyclicity of $\V_{\D}$), a contradiction.
     		
     		\item Let $\sigma_1=\gamma_2$. We claim that, each simplex of the trajectory $P$ is of the form $\gamma_2 \sqcup \tau$, for some $\tau \in \D$. For otherwise, there exists a simplex $\alpha_k$, for some $k \in \{1,\ldots,r\}$, such that $\alpha_k= \sigma_k \sqcup \tau_k$, and $\sigma_k \neq \gamma_2$. Since, the trajectory $P$ is closed, we consider $\alpha_k$ as $\alpha_1$, and rename rest of the simplices accordingly. Now, if $\sigma_k \neq \emptyset$, this subcase became similar to one of the afore-mentioned subcases (\ref{Sc1}) and (\ref{Sc2}), and if $\sigma_k= \emptyset$, this subcase became similar to the Case~2. Hence, each simplex of the trajectory $P$ is of the form $\gamma_2 \sqcup \tau$, for some $\tau \in \D$. Therefore, each simplex of the trajectory $P$ is in a one-to-one correspondence ($\gamma_2 \sqcup \tau \leftrightarrow \tau$) with the simplices of a $\V_{\D}$-trajectory, which can not be closed (from the acyclicity of $\V_{\D}$), a contradiction.
     		\end{enumerate}
            Hence, $\W$ is gradient vector field on $\C * \D$, which has critical simplices $\gamma_2 \sqcup \delta_2$ and $\emptyset \sqcup \delta_1$, with $\dim(\gamma_2 \sqcup \delta_2)=m+n-1$ and $\dim(\emptyset \sqcup \delta_1)=0$. This makes $\C*\D$ a combinatorial sphere.
            
             \end{proof}
      
    \begin{proof}[Proof of Theorem~\ref{t1.9}]
    	By proposition~\ref{p4.12}, $\s_1 * \s_2 * \cdots *\s_m$ is a $\mathbb{Z}_p$-combinatorial sphere, and the theorem follows from Theorem~\ref{t1.7}.
    \end{proof}

	\bibliographystyle{plain}
	\bibliography{Hopf-trace}  

    \appendix
    
    \section{Appendix} \label{Appendix}
    In this section, we prove the well-definedness of the notion of combinatorial degree. At the end of the section we show how the $\mathbb{Z}_p$-Tucker's lemma (Theorem~\ref{t1.8}) follows from degree version (Theorem~\ref{t1.9}).
    \begin{proposition}\label{A1}
    	Let $\s$ be a $d$-dimensional pseudomanifold, and let $\{\sigma_1, \sigma_2, \ldots, \sigma_n\}$ be the set of all $d$-dimensional simplices of $\s$. Let  $\beta=\sum_{i=1}^n s_i \sigma_i$, where $s_i \in \mathbb{Z}$, such that $\partial(\beta)=0$. Then,
    	 for all $i\neq j$, $|s_{i}|=|s_{j}|$, where $i,j\in \{1,\ldots,n\}$.
    \end{proposition}
    \begin{proof}
    	We have $\beta=\sum_{i=1}^n s_i\sigma_i$. We first consider the case where $\sigma_{i}$, $\sigma_{j}$ are `adjacent', i.e., $\sigma_{i} \cap \sigma_{j}$ is a $(d-1)$-dimensional simplex. Let $\sigma_{i} \cap \sigma_{j} =\tau$. Now, since $\s$ is a pseudomanifold, the only $d$-dimensional simplices that contain $\tau$ are $\sigma_{i}$ and $\sigma_{j}$. Hence, the coefficient of $\tau$ in $\partial(\beta)$ is
    	\[s_{i}\cdot[\sigma_{i},\tau]+s_{j} \cdot [\sigma_{j},\tau].\]
    	Since, $\partial(\beta)=0$, we have
    	\begin{align*}
    	&	s_{i}\cdot[\sigma_{i},\tau]+s_{j} \cdot [\sigma_{j},\tau]=0\\
    		\implies & 	|s_{i}|=	|s_{j}|  ~~\text{(since, $[\sigma_{i},\tau],[\sigma_{i},\tau]\in \pm1$)}
    	\end{align*}
    	
    	Now, when $\sigma_{i}$ and $\sigma_{j}$ are not adjacent, by definition of pseudomanifold, there exists a sequence of  maximal simplices
    	\[(\sigma_{i}=)~\sigma_{i_0}, \sigma_{i_1}, \ldots, \sigma_{i_m}~(=\sigma_{j}),\]
    	such that $\sigma_{i_k}\cap \sigma_{i_{k+1}}$ is a $(d-1)$-simplex, for all $k\in \{0,\ldots, m-1\}$ (i.e., two consecutive simplices in the sequence are adjacent). So, $|s_{i_k}|=|s_{i_{k+1}}|$, for $k\in \{0,\ldots, m-1\}$. Therefore, $|s_i|=|s_j|$.
    	
    \end{proof}
    \begin{theorem} \label{A2}
    	Every combinatorial sphere is orientable.
    \end{theorem}
    \begin{proof}
    	Let $\s$ be a combinatorial $d$-sphere, and $\sigma_1,\sigma_2, \ldots, \sigma_n$ be the $d$-dimensional simplices of $\s$. Let $\V$ be an gradient vector field on $\s$ having exactly two critical simplices, one $d$-dimensional and one $0$-dimensional. Let $\sigma_1$ be the $d$-dimensional critical simplex corresponding to $\V$. We have
    	 \begin{align*}
    		\overrightarrow{\sigma_1}^{(\V)} &=\sum_{\substack{P: P \text{ is a}\\ \V\text{-trajectory},\\ \inn(P)=\sigma_1}} w(P)\cdot \trr(P)\\
    		& = \sigma_1 + \sum_{i=2}^n \sum_{\substack{P: P \text{ is a}\\ \V\text{-trajectory,} \\ \inn(P)= \sigma_1, \\ \trr(P)=\sigma_i}} w(P)\cdot \sigma_i
    	\end{align*}
    	Note that, by Lemma~\ref{l4.7}, $\partial(\overrightarrow{\sigma_1}^{(\V)})=0$. Now, since the coefficient of $\sigma_1$ in $\overrightarrow{\sigma_1}^{(\V)}$ is 1 (by Observation~\ref{Ob2.13}(iv)). By Proposition~\ref{A1}, for all $i\in\{2,\ldots,n\}$, we have
    	\[|\sum_{\substack{P: P \text{ is a}\\ \V\text{-trajectory,} \\ \inn(P)= \sigma_1, \\ \trr(P)=\sigma_i}} w(P)|=1.\]
    	So, if we take $r_1=1$, and for all $i\in\{2,\ldots,n\}$, 
    		\[r_i=\sum_{\substack{P: P \text{ is a}\\ \V\text{-trajectory,} \\ \inn(P)= \sigma_1, \\ \trr(P)=\sigma_i}} w(P),\]
    		then $r=(r_1,r_2,\ldots, r_n)$ is an orientation on $\s$.
    \end{proof}
    \begin{proposition} \label{A3}
    	 There are exactly two possible orientations on a combinatorial sphere. More precisely, if $r$ and $t$ are two orientations on $\s$, then either $r=t$ or $r=-t$.
    \end{proposition}
    \begin{proof}
    Let $\s$ be a combinatorial $d$-sphere, and $\sigma_0,\sigma_1, \ldots, \sigma_n$ be the $n$-dimensional simplices of $\s$. Let, $r= (r_1,r_2,\ldots,r_n)$ and $t=(t_1,t_2,\ldots,t_n)$. We have
    \begin{align*}
    & [\s_r]= \sum_{i=1}^n r_i\sigma_i = r_1\sigma_1+\sum_{i=2}^n	r_i\sigma_i,\\
    \text{and } 	& [\s_t]= \sum_{i=1}^n t_i\sigma_i = t_1\sigma_1+\sum_{i=2}^n	t_i\sigma_i.
    \end{align*}
    
    \textbf{Case 1:} Let $t_1=r_1$. Then, we have
    \begin{align*}
    	[\s_r]- [\s_t] &= r_1\sigma_1+\sum_{i=2}^n	r_i\sigma_i- t_1\sigma_1-\sum_{i=2}^n	t_i\sigma_i\\
    	&=\sum_{i=2}^n	(r_i-t_i)\sigma_i
    \end{align*}
       Now, since $\partial([\s_r]- [\s_t])=0$, and the coefficient of $\sigma_1$ in 	$[\s_r]- [\s_t]$ is 0, by Proposition~\ref{A1}, we have 	$(r_i-t_i)=0$, i.e., $r_i=t_i$, for all $i\in \{2,\ldots,n\}$. Therefore, $r=t$.
       
        \textbf{Case 2:} Let $t_1=-r_1$. Then, we have
       \begin{align*}
       	[\s_r]+ [\s_t] &= r_1\sigma_1+\sum_{i=2}^n	r_i\sigma_i+ t_1\sigma_1+\sum_{i=2}^n	t_i\sigma_i\\
       	&=\sum_{i=1}^n	(r_i+t_i)\sigma_i
       \end{align*}
       Now, since $\partial([\s_r]+[\s_t])=0$, and the coefficient of $\sigma_1$ in 	$[\s_r]+[\s_t]$ is 0, by Proposition~\ref{A1}, we have 	$(r_i+t_i)=0$, i.e., $t_i=-r_i$, for all $i\in \{2,\ldots,n\}$. Therefore, $r=-t$.

    \end{proof}
    \begin{theorem} \label{A4}
    	 Let $\s$ be a combinatorial $d$-sphere and $r$ be an orientation on $\s$. Let $[\s_{r}]$ be the fundamental cycle corresponding to $r$. Let $\eta \in Z_d(\s,\mathbb{Z})$. Then
    \[\eta= k\cdot [\s_r]\]
    for some $k \in \mathbb{Z}$. 
    \end{theorem}
    \begin{proof}
    		Let $\{\sigma_1,\sigma_2,\ldots,\sigma_n\}$ be the set of all $d$-simplices of $\s$, and $\eta= \sum_{i=1}^n s_i \sigma_i$, where $s_i\in \mathbb{Z}$. By Proposition~\ref{A1}, $|s_i|=|s_j|$, for all $i,j \in \{1,2,\ldots,n\}$. Let $|s_1|=l$. When $l=0$, $\eta=0= 0\cdot [\s_r]$. 
    		
    		If $l\neq0$, we have
    		\[\eta  = \sum_{i=1}^n s_i \sigma_i= l \cdot \sum_{i=1}^n r^\prime_i \sigma_i,\]
    		where $r_i^\prime \in \pm 1$. Now, 
    		\begin{align*}&\partial(\eta)=0 \\
    			\implies& \partial(l \cdot \sum_{i=1}^n r^\prime_i \sigma_i)=0\\ \implies &l \cdot\partial( \sum_{i=0}^n r^\prime_i \sigma_i)=0\\ \implies&\partial( \sum_{i=1}^n r^\prime_i \sigma_i)=0 ~\text{(since, $l\neq 0$)}.\\
    		\end{align*}
    		Therefore, $r^\prime= (r^\prime_1,r^\prime_2,\ldots,r^\prime_n)$ is an orientation on $\s$. Hence, by Theorem~\ref{A3}, either $r=r^\prime$, or $r=-r^\prime$. When $r=r^\prime$, $\eta= l\cdot [\s_r]$, and when $r=-r^\prime$, $\eta= (-l)\cdot [\s_r]$.
    \end{proof}
    
    Now, we prove that the notion of the degree of a chain map (and the consequently the notion of combinatorial degree) is well-defined. More precisely, we have the following. 
    \begin{theorem} \label{A5}
    	Let $\s$ be a combinatorial $d$-sphere, and let $\phi_\#: C_\#(\s)\rightarrow C_\#(\s)$ be a chain map. Let $r_1$ and $r_2$ be two orientations on $\s$, such that, $\phi_d([\s_{r_1}])=k_1\cdot [\s_{r_1}]$ and $\phi_d([\s_{r_2}])=k_2\cdot [\s_{r_2}]$. Then 
    	$k_1=k_2$.
    	(Note that, the existence of the integers $k_1$ and $k_2$ is guaranteed from Theorem~\ref{A4})
    \end{theorem}
    \begin{proof}
   By Theorem~\ref{A3}, $r_1=r_2$, or $r_1=-r_2$. When $r_1=r_2$, we have $k_1=k_2$. When $r_1=-r_2$, we have
    \begin{align*}
    &[\s_{r_1}]=-[\s_{r_2}]\\
    \implies&	\phi_d([\s_{r_1}])=\phi_d(-[\s_{r_2}])\\
    \implies&	\phi_d([\s_{r_1}])=-\phi_d([\s_{r_2}])\\
    \implies&k_1\cdot [\s_{r_1}]=-k_2\cdot [\s_{r_2}]\\
     \implies&k_1\cdot [\s_{r_1}]=k_2\cdot (-[\s_{r_2}])\\
      \implies&k_1\cdot [\s_{r_1}]=k_2\cdot [\s_{r_1}]\\
      \implies & k_1=k_2.
    \end{align*}
    \end{proof}
    
   \begin{lemma}[cone lemma] \label{A6}
   	Let $\s$ be a combinatorial $d$-sphere, and $\T$ be a simplicial complex with dimension at most $d$. Let $r$ be an orientation on $\s$, and $[\s_r]$ be the fundamental cycle corresponding to $r$.  Let $\phi:x*\s \rightarrow \T$ be a simplicial map, and  \[\phi_\#: C_\#(x*\s, \mathbb{Z}) \rightarrow C_\#(\T,\mathbb{Z})\] be the chain map induced by $\phi$. Then,
   	\[\phi_d([\s_r])=0.\]
   \end{lemma}
   \begin{proof}
    Let, $\{\sigma_1,\sigma_2, \ldots, \sigma_n\}$ be the set of all $d$-dimensional simplices of $\s$. Let $r=(r_1,r_2,\ldots,r_n)$. We have, $[\s_r]=\sum_{i=1}^n r_i\sigma_i$. Now, from the definition of fundamental cycle, we have
    \begin{align}
    \nonumber	& \partial([\s_r])=0\\
    \nonumber	\implies & \partial(\sum_{i=1}^n r_i\sigma_i)=0\\
    \nonumber	\implies & \sum_{i=1}^n r_i \cdot \partial(\sigma_i)=0\\
    \nonumber	\implies & \sum_{i=1}^n r_i \cdot \left(\sum_{\tau \in S_{d-1}(\s)}[\sigma_i,\tau]\cdot \tau \right)=0\\
		\nonumber \implies &  \sum_{\tau \in S_{d-1}(\s)}\left(\sum_{i=1}^n r_i  \cdot[\sigma_i,\tau]\right)\cdot \tau=0 \\
		\implies & \sum_{i=1}^n r_i  \cdot[\sigma_i,\tau] =0, \text{ for all } \tau \in S_{d-1}(\s) \label{E20}
    \end{align}

   For an oriented simplex $\sigma=(v_0\ldots v_k)\in S_k(\s)$, let $x*\sigma=(x~v_0\ldots v_k)$. Now consider the following commutative diagram (Figure~\ref{f4}).
   
   	\begin{figure}[h]
   	\[\begin{tikzcd}
   		\cdots && {C_{d+1}(x*\s, \mathbb{Z})} && {C_d(x*\s, \mathbb{Z})}  && \cdots \\
   		\\
   		\cdots && 0 && {C_{d}(\T, \mathbb{Z})}  && \cdots
   		\arrow["{\partial_{d+2}}", from=1-1, to=1-3]
   		\arrow["{\partial_{d+1}}", from=1-3, to=1-5]
   		\arrow["{\phi_{d+1}}", from=1-3, to=3-3]
   		\arrow["{\partial_d}", from=1-5, to=1-7]
   		\arrow["{\phi_d}", from=1-5, to=3-5]
   		\arrow["{\partial_{d+2}}", from=3-1, to=3-3]
   		\arrow["{\partial_{d+1}}", from=3-3, to=3-5]
   		\arrow["{\partial_{d}}", from=3-5, to=3-7]   
   	\end{tikzcd}\]
   	\caption{Commutative diagram corresponding to the chain map $\phi_\#$.}\label{f4}
   \end{figure}
   
    By  the commutivity of the diagram (Figure~\ref{f4}), we have
   
   \begin{align}
   	\nonumber\implies  & \phi_d ( \partial_{d+1}(\sum_{i=1}^n r_i \cdot (x*\sigma_i)))= \partial_{d+1} ( \phi_{d+1}(\sum_{i=1}^n r_i \cdot (x*\sigma_i)))=\partial_{d+1} (0)=0\\
   	\implies  & \phi_d (\sum_{i=1}^n  r_i \cdot\partial_{d+1}(  x*\sigma_i))=0 \label{E21}
   \end{align}
  
    From the definition of the boundary map, we have
 	\begin{equation}\label{E22}
 	\partial_{d+1}(x* \sigma_i)= [x*\sigma_i, \sigma_i]\cdot \sigma_i + \sum_{\tau<\sigma_i} [x*\sigma_i, x*\tau]\cdot (x*\tau)
 	\end{equation}

 	Now observe that, if $\sigma=(v_0\ldots v_t \ldots v_d)$ and $\tau= (v_0\ldots \widehat{v_t} \ldots v_d )$, we have $[\sigma, \tau]= (-1)^t$ and $[x*\sigma, x*\tau]= (-1)^{t+1}$. In other words, for any $\sigma, \tau \in \s$, such that $\tau < \sigma$, we have
 	$[x*\sigma, x*\tau]=-[\sigma, \tau]$. Also note that, $[x*\sigma,\sigma]=(-1)^0=1$, for any $\sigma \in \s$.
 	So, the Equation~\eqref{E22} becomes,
 	\begin{align} \label{E23}
 		\partial_{d+1}(x* \sigma_i)= \sigma_i  - \sum_{\tau<\sigma_i} [\sigma_i, \tau]\cdot (x*\tau).
 	\end{align}
 	Now, putting this expression of $\partial_{d+1}(x* \sigma_i)$ (Equation~\eqref{E23}) into Equation~\eqref{E21}, we get
 	\begin{align*}
 	 & \phi_d (\sum_{i=1}^n  r_i \cdot( \sigma_i  - \sum_{\tau<\sigma_i} [\sigma_i, \tau]\cdot (x*\tau)))=0\\
 	 \implies& \phi_d ( \sum_{i=1}^n r_i \cdot \sigma_i  - \sum_{i=1}^n r_i \cdot \sum_{\tau<\sigma_i} [\sigma_i, \tau]\cdot (x*\tau))=0\\
 	  \implies &\phi_d ( \sum_{i=1}^n r_i \cdot \sigma_i  - \sum_{i=1}^n r_i \cdot \sum_{\tau \in S_{q-1}(\s)} [\sigma_i, \tau]\cdot (x*\tau))=0\\
 	   \implies& \phi_d (\sum_{i=1}^n r_i \cdot \sigma_i  - \sum_{\tau \in S_{q-1}(\s)} \sum_{i=1}^n r_i  \cdot [\sigma_i, \tau]\cdot (x*\tau))=0\\
 	    \implies& \phi_d (\sum_{i=1}^n r_i \cdot \sigma_i)=0~\text{ (from Equation~\eqref{E20})}\\
 	      \implies &\phi_d ([\s_r])=0
 	\end{align*}
 \end{proof}
 	
 	We now show that, the generalised $\mathbb{Z}_p$-Tucker's lemma (Theorem~\ref{t1.8}) follows from our combinatorial degree version (Theorem~\ref{t1.9}).
 	
 	\begin{proof} [Proof of Theorem~\ref{t1.8}]
 		We will prove the theorem by contradiction.
 		Let  $f: x* \Bd^k(S_1* \cdots * S_m) \rightarrow S_1 * \cdots *S_m$ be a simplicial map , such that,
 		\[f\big|_{\Bd^k(S_1* \cdots S_m)}:\Bd^k(S_1* \cdots S_m) \rightarrow S_1 * \cdots *S_m\]
 		is $\mathbb{Z}_p$-equivariant. Let $r$ be an orientation on $\Bd^k(S_1* \cdots S_m)$, and $[\Bd^k(S_1* \cdots S_m)_r]$ be the fundamental cycle corresponding to $r$. 
 		Let $f_\#: C_\#(x* \Bd^k(S_1* \cdots S_m)) \rightarrow C_\#(S_1 * \cdots *S_m)$ be chain map corresponding to $f$. Then from cone lemma (Theorem~\ref{A6}), we have
 		\begin{align*}
 			& f_d([\Bd^k(S_1* \cdots S_m)_r])=0\\
 			\implies& (f\big|_{\Bd^k(S_1* \cdots S_m)})_d([\Bd^k(S_1* \cdots S_m)_r])=0\\
 			\implies & \deg (f\big|_{\Bd^k(S_1* \cdots S_m)})=0.
 		\end{align*}
 		But, by Theorem~\ref{t1.9}, we have $\deg (f\big|_{\Bd^k(S_1* \cdots S_m)})\equiv 1 \pmod{p}$, a contradiction.
 	\end{proof}

 \end{document}